\documentclass{amsart}
\usepackage{amsfonts, amsbsy, amsmath, amssymb}

\usepackage{scrextend}

\makeatletter
\@namedef{subjclassname@2010}{%
  \textup{2020} Mathematics Subject Classification}
\makeatother

\ifx\fiverm\undefined 
  \newfont\fiverm{cmr5} 
\fi 

\input prepictex.tex
\input pictex.tex
\input postpictex.tex

\newtheorem{thm}{Theorem}[section]
\newtheorem{lem}[thm]{Lemma}

\newtheorem{exmp}[thm]{Example}

\newtheorem{algo}[thm]{Algorithm}
\newtheorem{rmk}[thm]{Remark}

\newtheorem{ques}[thm]{Question}
\newtheorem{thm-con}[thm]{Theorem-Conjecture}
\numberwithin{equation}{section}

\theoremstyle{definition}
\newtheorem{defn}[thm]{Definition}

\allowdisplaybreaks

\newcommand{\f}{\Bbb F}

\makeatletter
\def\bbordermatrix#1{\begingroup \m@th
  \@tempdima 4.75\p@
  \setbox\z@\vbox{%
    \def\cr{\crcr\noalign{\kern2\p@\global\let\cr\endline}}%
    \ialign{$##$\hfil\kern2\p@\kern\@tempdima&\thinspace\hfil$##$\hfil
      &&\quad\hfil$##$\hfil\crcr
      \omit\strut\hfil\crcr\noalign{\kern-\baselineskip}%
      #1\crcr\omit\strut\cr}}%
  \setbox\tw@\vbox{\unvcopy\z@\global\setbox\@ne\lastbox}%
  \setbox\tw@\hbox{\unhbox\@ne\unskip\global\setbox\@ne\lastbox}%
  \setbox\tw@\hbox{$\kern\wd\@ne\kern-\@tempdima\left[\kern-\wd\@ne
    \global\setbox\@ne\vbox{\box\@ne\kern2\p@}%
    \vcenter{\kern-\ht\@ne\unvbox\z@\kern-\baselineskip}\,\right]$}%
  \null\;\vbox{\kern\ht\@ne\box\tw@}\endgroup}
\makeatother

\begin{document}

\title{A General Construction of Permutation Polynomials of $\f_{q^2}$}

\author[Xiang-dong Hou]{Xiang-dong Hou}
\address{Department of Mathematics and Statistics,
University of South Florida, Tampa, FL 33620}
\email{xhou@usf.edu}

\author[Vincenzo Pallozzi Lavorante]{Vincenzo Pallozzi Lavorante *}
\address{Universit\`a degli Studi di Modena e Reggio Emilia, Italy}
\email{vincenzo.pallozzilavorante@unimore.it}
\thanks{* The research of Vincenzo Palozzi Lavorante was partially supported  by the Italian National Group for Algebraic and Geometric Structures and their Applications (GNSAGA - INdAM)}

\keywords{
finite field, permutation polynomial, self-dual polynomial
}

\subjclass[2010]{11T06, 11T30, 11T55}

\begin{abstract}
Let $r$ be a positive integer, $h(X)\in\Bbb F_{q^2}[X]$, and $\mu_{q+1}$ be the subgroup of order $q+1$ of $\Bbb F_{q^2}^*$. It is well known that $X^rh(X^{q-1})$ permutes $\Bbb F_{q^2}$ if and only if $\text{gcd}(r,q-1)=1$ and $X^rh(X)^{q-1}$ permutes $\mu_{q+1}$. There are many ad hoc constructions of permutation polynomials of $\Bbb F_{q^2}$ of this type such that $h(X)^{q-1}$ induces monomial functions on the cosets of a subgroup of $\mu_{q+1}$. We give a general construction that can generate, through an algorithm, {\em all} permutation polynomials of $\Bbb F_{q^2}$ with this property, including many which are not known previously. The construction is illustrated explicitly for permutation binomials and trinomials.
\end{abstract}

\maketitle

\section{Introduction}

Let $\f_q$ denote the finite field with $q$ elements. A polynomial $f(X)\in\f_q[X]$ is called a {\em permutation polynomial} (PP) of $\f_q$ if it induces a permutation of $\f_q$. Let $r$ be a positive integer, $d\mid q-1$, and $h(X)\in\f_q[X]$. It is well known \cite{Park-Lee-BAMS-2001, Wang-LNCS-2007, Zieve-PAMS-2009} that $X^rh(X^{(q-1)/d})$ is a PP of $\f_q$ if and only if $\text{gcd}(r,(q-1)/d)=1$ and $X^rh(X)^{(q-1)/d}$ permutes the multiplicative group $\mu_d:=\{x\in\f_q^*:x^d=1\}$. (In general, we use $\mu_m$ to denote a multiplicative group of order $m$ of a finite field.) Replacing $q$ with $q^2$ and $d$ with $q+1$, we see that for $h(X)\in\f_{q^2}[X]$, $X^rh(X^{q-1})$ is a PP of $\f_{q^2}$ if and only $\text{gcd}(r,q-1)=1$ and $X^rh(X)^{q-1}$ permutes $\mu_{q+1}$. To facilitate the constructions of $\mu_{q+1}$ of the form $X^rh(X)^{q-1}$, the following idea has been used by several authors \cite{Cao-Hou-Mi-Xu-FFA-2020, Lavorante-arXiv:2105.12012, Li-Qu-Chen-Li-CC-2018, Qin-Yan-AAECC-2021, Zheng-Yuan-Yu-FFA-2019}: Let $H$ be a subgroup of $\mu_{q+1}$ of small index. Construct a polynomial $h(X)\in\f_{q^2}[X]$ such that $h(X)^{q-1}$ induces monomial functions on each coset of $H$ in $\mu_{q+1}$. With such a property, $X^rh(X)^{q-1}$ permutes $\mu_{q+1}$ if and only if some simple number theoretic conditions on the parameters are satisfied. This method has produced many results. However, these results only deal with specific situations, leaving a unified treatment to be desired.

In the present paper, we take a general approach to the question. The main result is an algorithm (Algorithm~\ref{Algo2.4}) that produces {\em all} PPs of $\f_{q^2}$ of the form $X^rh(X^{q-1})$ such that $h(X)^{q-1}$ induces monomial functions on the cosets of a subgroup in $\mu_{q+1}$. Let $d\mid q+1$ and $\epsilon\in\f_{q^2}^*$ be such that $o(\epsilon)=d$. Define
\begin{equation}\label{1.1}
A_k=\{x\in\mu_{q+1}:x^{(q+1)/d}=\epsilon^k\},\quad 0\le k<d.
\end{equation}
Then $A_0=\mu_{(q+1)/d}$, and $A_0,\dots,A_{d-1}$ are the cosets of $\mu_{(q+1)/d}$ in $\mu_{q+1}$, whence
\begin{equation}\label{1.2}
\mu_{q+1}=\bigsqcup_{k=0}^{d-1}A_k.
\end{equation}
Since $X^{n_1(q-1)}\equiv X^{n_2(q-1)}\pmod{X^{q^2-1}-1}$ whenever $n_1\equiv n_2\pmod{q+1}$, it suffices to consider $h\in\f_{q^2}[X]$ with $\deg h\le q$. Write
\begin{equation}\label{1.3}
h(X)=\sum_{\substack{0\le i<(q+1)/d\cr 0\le j<d}}a_{ij}X^{i+j(q+1)/d}.
\end{equation}
The objective is to find conditions on $a_{ij}\in\f_{q^2}$ such that for every $0\le k<d$, 
\begin{equation}\label{1.4}
x^rh(x)^{q-1}=\lambda_kx^{e_k}\ \text{for all}\ x\in A_k,
\end{equation}
where $e_k\in\Bbb Z$ and $\lambda_k\in\mu_{q+1}$, say $\lambda_k\in A_{\pi(k)}$.

\begin{thm}\label{T1.1}
Assume that \eqref{1.4} is satisfied for all $0\le k<d$. Then $X^rh(X)^{q-1}$ permutes $\mu_{q+1}$ if and only if 
\[
\text{\rm gcd}\Bigl(e_k,\frac{q+1}d\Bigr)=1,\quad 0\le k<d,
\] 
and
\[
k\mapsto \pi(k)+e_kk
\]
is a permutation of $\Bbb Z/d\Bbb Z$.
\end{thm}

\begin{proof}
By \eqref{1.4}, $X^rh(X)^{q-1}$ maps $A_k$ to $A_{\pi(k)+e_kk}$. This map is one-to-one on $A_k$ if and only if $\text{gcd}(e_k,(q+1)/d)=1$. Hence the conclusion.
\end{proof}

Therefore, the crucial question is to determine the polynomials $h(X)$ satisfying \eqref{1.4}. In Section 2, we will resolve this question and we will describe an algorithm that produces all PPs of the form $X^rh(X^{q-1})$ of $\f_{q^2}$ satisfying \eqref{1.4}. In Section 3, we determine all permutation binomials of $\f_{q^2}$ resulting from this algorithm and it turns out that these permutation binomials were all known previously. In Section 4, we determine all permutation trinomials of $\f_{q^2}$ resulting from the algorithm. There are four classes such permutation trinomials, excluding those that were previously known. These four classes, in their generality, appear to be new, although many special cases have been discovered by other authors. Additional examples of the algorithm are given in Section 5. Overall, this approach reveals many PPs that were not known previously. 

\medskip

\noindent{\bf Remark.} In the present paper, we investigate polynomials $h(X)\in\f_{q^2}[X]$ such that $h(X)^{q-1}$ induces monomial functions on the cosets of a subgroup of $\mu_{q+1}$ and permutes $\mu_{q+1}$ as a whole. Before this approach became popular in recent years, people had explored a similar method for PPs of $\f_q$. Several authors \cite{Fernando-Hou-FFA-2012, Niederreiter-Winterhof-DM-2005, Wang-LNCS-2007, Wang-FFA-2013, Zha-Hu-FFA-2012} had studied PPs of $\f_q$ which induce monomial functions on the cosets of a subgroup of $\f_q^*$.  

\section{The Construction}

For $a\in\f_{q^2}$, define $\bar a=a^q$; for $f(X)=\sum_{i=0}^na_iX^i\in\f_{q^2}[X]$ with $a_n\ne 0$, define
\[
\bar f(X)=\sum_{i=0}^n\bar a_iX^i
\]
and
\[
\tilde f(X)=X^n\bar f(X^{-1})=\sum_{i=0}^n\bar a_iX^{n-i}.
\]
Obviously, $\bar{\bar f}=f$ and $\tilde{\tilde f}=f$. If $\tilde f=cf$ for some $c\in\f_{q^2}^*$, $f$ is said to be {\em self-dual}; in this case, it is necessary that $c\in\mu_{q+1}$. Self-dual polynomials were first introduced in \cite{Hou-pp} for a different purpose; they will also play an important role in the present paper.

We follow the notation of Section~1. Let $h(X)$ be given in \eqref{1.3} and assume that $h$ has no root in $\mu_{q+1}$. For $x\in A_k$, where $0\le k<d$, we have 
\begin{equation}\label{2.1}
h(X)=\sum_{i,j}a_{ij}\epsilon^{jk}x^i=\sum_iM_{ik}x^i,
\end{equation}
where
\begin{equation}\label{2.1.1}
M_{ik}=\sum_ja_{ij}\epsilon^{jk}.
\end{equation}
Note that the $((q+1)/d)\times d$ matrices $[M_{ik}]$ and $[a_{ij}]$ are related by the $d\times d$ Vandermonde matrix $[\epsilon^{jk}]$:
\[
[M_{ik}]=[a_{ij}]\,[\epsilon^{jk}],\qquad [a_{ij}]=\frac 1d\,[M_{ik}]\,[\epsilon^{-kj}].
\]
By \eqref{2.1}, for $x\in\mu_{q+1}$,
\begin{equation}\label{2.2}
x^rh(x)^{q-1}=x^r\frac{h(x)^q}{h(x)}=x^r\,\frac{\displaystyle\sum_i\overline M_{ik}\,x^{-i}}{\displaystyle\sum_iM_{ik}\,x^i}.
\end{equation}
Write
\begin{equation}\label{2.3}
\sum_iM_{ik}X^i=X^sL(X),
\end{equation}
where $L(X)\in\f_{q^2}[X]$, $L(0)\ne 0$, $\deg L=t$, $s+t<(q+1)/d$. Then \eqref{2.3} becomes
\begin{equation}\label{2.4}
x^rh(x)^{q-1}=x^r\,\frac{x^{-s}\bar L(x^{-1})}{x^sL(x)}=x^{r-2s-t}\frac{\tilde L(x)}{L(x)}.
\end{equation}

The following lemma is crucial.

\begin{lem}\label{L2.1}
Let $L(X)\in\f_{q^2}[X]$ be such that $L(0)\ne 0$, $\deg L=t<(q+1)/d$, and $L(X)$ has no root in $A_k$.
\begin{itemize}
\item[(i)]
Assume that there exist $0\le \tau<(q+1)/d$ and $\lambda\in\mu_{q+1}$ such that
\begin{equation}\label{2.5}
\frac{\tilde L(x)}{L(x)}=\lambda x^\tau\quad \text{for all}\ x\in A_k.
\end{equation}
Then either $\tau=0$ or $(q+1)/d-t\le \tau\le t$.

\item[(ii)]
When $\tau=0$, \eqref{2.5} is satisfied if and only if 
\begin{equation}\label{2.6}
\tilde L(X)=\lambda L(X).
\end{equation}

\item[(iii)]
When $(q+1)/d-t\le \tau\le t$, \eqref{2.5} is satisfied if and only if 
\begin{equation}\label{2.7}
L(X)=P(X)+X^{(q+1)/d-\tau}Q(X),
\end{equation}
where $P,Q\in\f_{q^2}[X]$, $\deg P=t-\tau$, $\tilde P=\lambda P$, $\deg Q=\tau+t-(q+1)/d$, $\tilde Q=\lambda\epsilon^k Q$.
\end{itemize}
\end{lem}

\begin{proof}
(ii) Since $\deg(\tilde L-\lambda L)\le t<(q+1)/d=|A_k|$, 
\[
\tilde L(x)-\lambda L(x)\ \text{for all}\ x\in A_k\ \Leftrightarrow\ \tilde L(X)-\lambda L(X)=0.
\]

\medskip
(iii) ($\Rightarrow$) We have $X^{(q+1)/d}-\epsilon^k \mid \tilde L(X)-\lambda X^\tau L(X)$, say
\begin{equation}\label{2.8}
\tilde L(X)-\lambda X^\tau L(X)=g(X)(\epsilon^k-X^{(q+1)/d}),
\end{equation}
where $g(X)\in\f_{q^2}[X]$ with $\deg g=\tau+t-(q+1)/d$. In \eqref{2.8},
\begin{align*}
\widetilde{\tilde L-\lambda X^\tau L}\,&=X^{\tau+t}\bigl(\bar{\tilde L}(X^{-1})-\bar\lambda X^{-\tau}\bar L(X^{-1})\bigr)\cr
&=X^\tau\tilde{\tilde L}-\bar\lambda\tilde L\cr
&=X^\tau L-\bar\lambda\tilde L\cr
&=-\bar\lambda(\tilde L-\lambda X^\tau L).
\end{align*}
Hence
\[
\widetilde{g(X)(\epsilon^k-X^{(q+1)/d})}=-\bar\lambda g(X)(\epsilon^k-X^{(q+1)/d}),
\]
i.e.,
\[
\tilde g(X)(\epsilon^{-k}X^{(q+1)/d}-1)=-\bar\lambda g(X)(\epsilon^k-X^{(q+1)/d}),
\]
hence $\tilde g=\bar\lambda \epsilon^kg$. Therefore, \eqref{2.8} becomes
\begin{equation}\label{2.9}
\tilde L-\lambda X^\tau L=\lambda\tilde g-X^{(q+1)/d}g.
\end{equation}
Let $L=a_0+\cdots+a_tX^t$ and $g=b_0+\cdots+b_vX^v$, where $v=\deg g=\tau+t-(q+1)/d$. The coefficients of $\tilde L-\lambda X^\tau L$ and $\lambda\tilde g-X^{(q+1)/d}g$ are illustrated in Figure~\ref{F1}. It follows from \eqref{2.9} and Figure~\ref{F1} that $a_i=0$ for $t-\tau<i<t-v$, $a_{t-\tau}\ne 0$, and 
\begin{align*}
L\,&=a_0+\cdots+a_{t-\tau}X^{t-\tau}+X^{(q+1)/d-\tau}(a_{t-v}+\cdots+a_tX^v)\cr
&=P(X)+X^{(q+1)/d-\tau}Q(X),
\end{align*}
where $P(X)=a_0+\cdots+a_{t-\tau}X^{t-\tau}$, which satisfies $\tilde P=\lambda P$, and $Q(X)=a_{t-v}+\cdots+a_tX^v=-\bar\lambda g(X)$, which satisfies $\tilde Q=-\lambda\tilde g=-\lambda\bar\lambda\epsilon^k g=\lambda\epsilon^kQ$.

\begin{figure}
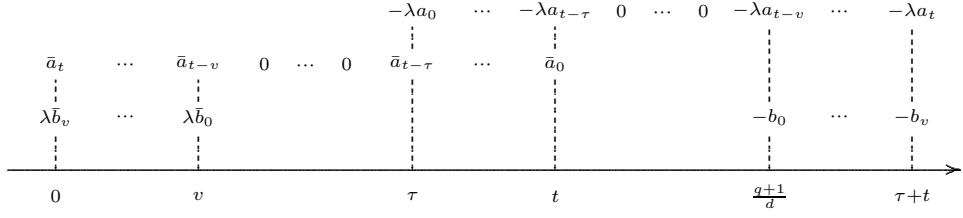

\[
\beginpicture
\setcoordinatesystem units <1.8em, 2em> point at 0 0

\arrow <5pt> [.2,.67] from -10 0 to 10 0 

\setdashes<0.18em>
\plot 9 0   9 0.7 /
\plot 9 1.3   9 2.7 /
\plot 6 0   6 0.7 /
\plot 6 1.3   6 2.7 /
\plot -9 0   -9 0.7 /
\plot -9 1.3   -9 1.7 /
\plot -6 0   -6 0.7 /
\plot -6 1.3   -6 1.7 /
\plot -1.5 0   -1.5 1.7 /
\plot -1.5 2.3   -1.5 2.7 /
\plot 1.5 0   1.5 1.7 /
\plot 1.5 2.3   1.5 2.7 /

\put {$\scriptstyle 0$}  at -9 -0.5
\put {$\scriptstyle \tau+t$}  at 9 -0.5
\put {$\scriptstyle \frac{q+1}d$}  at 6 -0.5
\put {$\scriptstyle -b_0$}  at 6 1
\put {$\scriptstyle -\lambda a_{t-v}$}  at 6 3
\put {$\scriptstyle -b_v$}  at 9 1
\put {$\scriptstyle -\lambda a_t$}  at 9 3
\put {$\scriptstyle \cdots$}  at 7.5 3
\put {$\scriptstyle \cdots$}  at 7.5 1
\put {$\scriptstyle v$}  at -6 -0.5
\put {$\scriptstyle \lambda \bar b_v$}  at -9 1
\put {$\scriptstyle \lambda \bar b_0$}  at -6 1
\put {$\scriptstyle \tau$}  at -1.5 -0.5
\put {$\scriptstyle t$}  at 1.5 -0.5
\put {$\scriptstyle \bar a_{t-\tau}$}  at -1.5 2
\put {$\scriptstyle \bar a_0$}  at 1.5 2
\put {$\scriptstyle -\lambda a_0$}  at -1.5 3
\put {$\scriptstyle -\lambda a_{t-\tau}$}  at 1.5 3
\put {$\scriptstyle \cdots$}  at 0 3
\put {$\scriptstyle \cdots$}  at 0 2
\put {$\scriptstyle 0\kern1em \cdots \kern1em 0$}  at 3.75 3
\put {$\scriptstyle 0\kern1em \cdots \kern1em 0$}  at -3.75 2
\put {$\scriptstyle \bar a_t$}  at -9 2
\put {$\scriptstyle \bar a_{t-v}$}  at -6 2
\put {$\scriptstyle \cdots$}  at -7.5 2
\put {$\scriptstyle \cdots$}  at -7.5 1

\endpicture
\]
\caption{The coefficients of $-\lambda X^\tau L$, $\tilde L$ and $\lambda\tilde g-X^{(q+1)/d}g$ (from top to bottom)}\label{F1}
\end{figure}

\medskip
($\Leftarrow$) We have
\begin{align*}
\tilde L-\lambda X^\tau L\,&=(\widetilde{P+X^{(q+1)/d-\tau}Q})-\lambda X^\tau(P+X^{(q+1)/d-\tau}Q)\cr
&=X^t\bigl(\bar P(X^{-1})+X^{-((q+1)/d-\tau)}\bar Q(X^{-1})\bigr)-\lambda X^\tau(P+X^{(q+1)/d-\tau}Q)\cr
&=X^t\bar P(X^{-1})+X^v\bar Q(X^{-1})-\lambda X^\tau(P+X^{(q+1)/d-\tau}Q)\cr
&=X^\tau\tilde P+\tilde Q-\lambda X^\tau(P+X^{(q+1)/d-\tau}Q)\cr
&=X^\tau\lambda P+\lambda\epsilon^k Q-\lambda X^\tau(P+X^{(q+1)/d-\tau}Q)\cr
&=\lambda Q(\epsilon^k-X^{(q+1)/d}).
\end{align*}
Hence
\[
\frac{\tilde L(x)}{L(x)}=\lambda x^\tau\quad \text{for all}\ x\in A_k.
\]

\medskip
(i) Assume $\tau>0$. By the proof of (iii) ($\Rightarrow$), $\tau+t-(q+1)/d=\deg g\ge 0$, whence $\tau\ge(q+1)/d-t$. It remains to show that $\tau\le t$. Assume to the contrary that $\tau>t$. Then Figure~\ref{F1} is replaced by Figure~\ref{F2}. Then $a_0=0$, which is a contradiction.
\begin{figure}
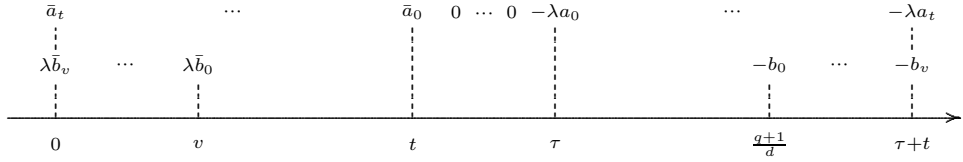

\[
\beginpicture
\setcoordinatesystem units <1.8em, 2em> point at 0 0

\arrow <5pt> [.2,.67] from -10 0 to 10 0 

\setdashes<0.18em>
\plot 9 0   9 0.7 /
\plot 9 1.3   9 1.7 /
\plot 6 0   6 0.7 /
\plot -9 0   -9 0.7 /
\plot -9 1.3   -9 1.7 /
\plot -6 0   -6 0.7 /
\plot -1.5 0   -1.5 1.7 /
\plot 1.5 0   1.5 1.7 /

\put {$\scriptstyle 0$}  at -9 -0.5
\put {$\scriptstyle \tau+t$}  at 9 -0.5
\put {$\scriptstyle \frac{q+1}d$}  at 6 -0.5
\put {$\scriptstyle -b_0$}  at 6 1
\put {$\scriptstyle -b_v$}  at 9 1
\put {$\scriptstyle -\lambda a_t$}  at 9 2
\put {$\scriptstyle \cdots$}  at 7.5 1
\put {$\scriptstyle v$}  at -6 -0.5
\put {$\scriptstyle \lambda \bar b_v$}  at -9 1
\put {$\scriptstyle \lambda \bar b_0$}  at -6 1
\put {$\scriptstyle t$}  at -1.5 -0.5
\put {$\scriptstyle \tau$}  at 1.5 -0.5
\put {$\scriptstyle \bar a_0$}  at -1.5 2
\put {$\scriptstyle -\lambda a_0$}  at 1.5 2
\put {$\scriptstyle 0\kern0.5em \cdots \kern0.5em 0$}  at 0 2
\put {$\scriptstyle \cdots$}  at -5.25 2
\put {$\scriptstyle \cdots$}  at 5.25 2
\put {$\scriptstyle \bar a_t$}  at -9 2
\put {$\scriptstyle \cdots$}  at -7.5 1

\endpicture
\]
\caption{The coefficients of $\tilde L-\lambda X^\tau L$ (top) and $\lambda\tilde g-X^{(q+1)/d}g$ (bottom)}\label{F2}
\end{figure}
\end{proof}

\begin{defn}\label{D2.2}
Let $0\le k<d$, $0\le t<(q+1)/d$ and $\lambda\in\mu_{q+1}$. Define
\begin{equation}\label{2.10}
\mathcal L_k(t,0;\lambda)=\{L\in\f_{q^2}[X]:\deg L=t,\ \tilde L=\lambda L,\ \text{gcd}(L,X^{(q+1)/d}-\epsilon^k)=1\},
\end{equation}
and for $(q+1)/d-t\le \tau\le t$, define
\begin{align}\label{2.11}
\mathcal L_k(t,\tau;\lambda)=\{& L=P+X^{(q+1)/d-\tau}Q: P,Q\in\f_{q^2}[X],\\
&\deg P=t-\tau,\ \tilde P=\lambda P,\ \deg Q=\tau+t-(q+1)/d,\cr
&\tilde Q=\lambda\epsilon^k Q,\ \text{gcd}(L,X^{(q+1)/d}-\epsilon^k)=1\}.\nonumber
\end{align}
\end{defn}

It follows from \eqref{2.3}, \eqref{2.4} and Lemma~\ref{L2.1} that $X^rh(X)^{q-1}$ is a monomial function on $A_k$ if and only if there exist $s,t\ge 0$ with $s+t<(q+1)/d$, $\lambda\in\mu_{q+1}$, and integer $\tau\in\{0\}\cup[(q+1)/d-t,t]$ such that $\sum_iM_{ik}X^i=X^sL(X)$, where $L\in\mathcal L_k(t,\tau;\lambda)$. When this happens,
\begin{equation}\label{2.13}
x^rh(x)^{q-1}=\lambda x^{r-2s-t+\tau}\quad \text{for all}\ x\in A_k.
\end{equation}
Combining the above statement with Theorem~\ref{T1.1}, we obtain the main theorem of the paper:

\begin{thm}\label{T2.3}
Let $h(X)$ be given by \eqref{1.3} and $[M_{ik}]$ be given by \eqref{2.1.1}. Then $X^rh(X^{q-1})$ is a PP of $\f_{q^2}$ such that $X^rh(X)^{q-1}$ is a monomial function on $A_k$ for every $0\le k<d$ if and only if the following conditions are satisfied.
\begin{itemize}
\item[(i)]
For each $0\le k<d$, there exist $s_k,t_k\ge 0$ with $s_k+t_k<(q+1)/d$, $\pi(k)\in\Bbb Z/d\Bbb Z$, $\lambda_k\in A_{\pi(k)}$ and $\tau_k\in\{0\}\cup[(q+1)/d-t_k,t_k]$ such that $\sum_iM_{ik}X^i=X^{s_k}L_k(X)$, where $L_k\in\mathcal L_k(t_k,\tau_k;\lambda_k)$.

\item[(ii)]
$\text{\rm gcd}(r,q-1)=1$ and $\text{\rm gcd}(e_k,(q+1)/d)=1$ for all $0\le k<d$, where
\[
e_k=r-2s_k-t_k+\tau_k.
\]

\item[(iii)]
The map $k\mapsto \pi(k)+e_kk$ permutes $\Bbb Z/d\Bbb Z$.
\end{itemize}
\end{thm}

Theorem~\ref{T2.3} can be stated as an algorithm.

\begin{algo}\label{Algo2.4}\rm
Let $r$ be a positive integer such that $\text{gcd}(r,q-1)=1$ and let $d\mid q+1$.
\begin{labeling}{{\bf Output:}}
\item [{\bf Input:}] Sequences $s_k$, $t_k$, $\tau_k$, $\pi(k)$, $\lambda_k$, $0\le k<d$, described below.
\medskip

\item [{\bf Output:}] A PP of $\f_{q^2}$ of the form $X^rh(X^{q-1})$ such that $X^rh(X)^{q-1}$ is a monomial function on each $A_k$, $0\le k<d$.
\medskip

\item [{\bf Note:}] All PPs of $\f_{q^2}$ with such properties can be produced by this algorithm.
\medskip

\item [{\bf Step 1:}] Choose integer sequence $s_k,t_k,\tau_k\ge $, $0\le k<d$, such that $s_k+t_k<(q+1)/d$, $\tau_k\in\{0\}\cup[(q+1)/d-t_k,t_k]$, and $e_k:=r-2s_k-t_k+\tau_k$ satisfies $\text{gcd}(e_k,(q+1)/d)=1$.

\medskip

\item [{\bf Step 2:}] Choose a sequence $\pi(k)\in\Bbb Z/d\Bbb Z$, $0\le k<d$, such that $k\mapsto \pi(k)+e_kk$ permutes $\Bbb Z/d\Bbb Z$.

\medskip

\item [{\bf Step 3:}] For each $0\le k<d$, choose $\lambda_k\in A_{\pi(k)}$ and $L_k\in\mathcal L_k(t_k,\tau_k;\lambda_k)$.
\medskip

\item [{\bf Step 4:}] Compute the $((q+1)/d)\times d$ matrix $[M_{ik}]$ such that 
\[
X^{s_k}L_k=\sum_iM_{ik}X^i,
\]
and compute the $((q+1)/d)\times d$ matrix
\[
[a_{ij}]=\frac 1d[M_{ik}][\epsilon^{-kj}].
\]

\medskip

\item [{\bf Step 5:}] Let 
\[
h(X)=\sum_{i,j}a_{ij}X^{i+j(q+1)/d}.
\]
Then $X^rh(X^{q-1})$ is the output PP of $\f_{q^2}$.
\end{labeling}
\end{algo}

\begin{rmk}\label{R2.5}\rm
In Step 3, when choosing $L_k\in\mathcal L_k(t_k,\tau_k;\lambda_k)$, it is required that $\text{gcd}(L_k,X^{(q+1)/d}-\epsilon^k)=1$. However, this condition is automatically satisfied if $h(X)$ in Step 5 satisfies $\text{gcd}(h,X^{q+1}-1)=1$. In fact, $\text{gcd}(L_k,X^{(q+1)/d}-\epsilon^k)=1$ for all $0\le k<d$ if and only if $\text{gcd}(h,X^{q+1}-1)=1$. 
\end{rmk}

There are two ways to use this algorithm: forward or backward. In the forward approach, we simply proceed from Step 1 through Step 5. The advantage of this approach is that there are few restrictions on the choices of the sequences; the drawback is that we have little control over the appearance of the resulting PP. A few examples of the forward approach are given in Section 5. In the backward approach, we first impose conditions on $[a_{ij}]$. (For example, we may require $h(X)$ to be a binomial of a trinomial.) We then compute $[M_{ik}]$ and determine if the sequences $L_k$, $s_k$, $t_k$, $\tau_k$, $\pi(k)$, $\lambda_k$ exist. The benefit of this approach is that we have more control over the appearance of the resulting PP. However, the conditions for the aforementioned sequence to be existent could be complicated. In Sections 3 and 4, we use the backward approach to determine the permutation binomials and trinomials obtainable from the algorithm. 

For $0\le t<(q+1)/d$, $\tau\in\{0\}\cup[(q+1)/d-t,t]$, $\lambda\in\mu_{q+1}$ and $0\le k<d$, write $\lambda=a^{1-q}$, where $a\in\f_{q^2}^*$, and $\epsilon^k=b^{(q+1)/d}$, where $b\in\mu_{q+1}$. Then it is easy to see that the map
\[
\begin{array}{ccc}
\mathcal L_k(t,\tau;\lambda)&\longrightarrow&\mathcal L_0(t,\tau;1)\vspace{0.2em}\cr
L(X)&\longmapsto& aL(bX)
\end{array}
\]
is a bijection. Set $l(t,\tau)=|\mathcal L_0(t,\tau;1)|$. Then $|\mathcal L_k(t,\tau;\lambda)|=l(t,\tau)$, which is independent of $k$ and $\lambda$.

Let 
\begin{align*}
\Omega=\{(s,t,\tau)\in\Bbb N^3:\; & s+t<(q+1)/d,\ \tau\in\{0\}\cup[(q+1)/d-t,t],\cr
&\text{gcd}(r-2s-t+\tau,(q+1)/d)=1\}.
\end{align*}
In Step 2 of Algorithms~\ref{Algo2.4}, the number of choices for the sequence $\pi(k)$ is $d!$. In Step~3, the number of choices for $\lambda_k$ is $(q+1)/d$ and the number of choices for $L_k$ is $l(t_k,\tau_k)$. Therefore, the total number of PPs produced by the algorithm is
\begin{align}\label{2.15}
&\sum_{(s_0,t_0,\tau_0),\dots,(s_{d-1},t_{d-1},\tau_{d-1})\in\Omega}d!\prod_{k=0}^{d-1}\Bigl(\frac{q+1}dl(t_k,\tau_k)\Bigr)\\
&=d!\Bigl(\frac{q+1}d\Bigr)^d\Bigl(\sum_{(s,t,\tau)\in\Omega}l(t,\tau)\Bigr)^d\cr
&=d!\Bigl(\frac{q+1}d\Bigr)^d\Bigl(\sum_{\substack{0\le t<(q+1)/d\cr \tau\in\{0\}\cup[(q+1)/d-t,t]}}m(t,\tau)l(t,\tau)\Bigr)^d,\nonumber
\end{align}
where
\begin{equation}\label{2.16}
m(t,\tau)=|\{(0\le s<(q+1)/d-t:\text{gcd}(r-2s-t+\tau,(q+1)/d)=1\}|.
\end{equation}
When $\tau=0$, $l(t,0)$ is determined by the following lemma.

\begin{lem}\label{L2.5]}
For $0\le t<(q+1)/d$,
\[
l(t,0)=(q^2-1)\sum_{i=0}^{t-1}(-1)^i\binom{(q+1)/d}i q^{t-i-1}+(-1)^t(q-1)\binom{(q+1)/d}t.
\]
\end{lem}

\begin{proof}
Let $\Lambda_t$ denote the number of monic self-dual polynomials of degree $t$ in $\f_{q^2}[X]$. It is known that \cite{Hou-pp}
\[
\Lambda_t=\begin{cases}
1&\text{if}\ t=0,\cr
(q+1)q^{t-1}&\text{if}\ t>0.
\end{cases}
\]
For $Y\subset\mu_{q+1}$, let
\[
\mathcal L_Y=\Bigl\{L\in\f_{q^2}[X]\ \text{monic, self-dual},\ \deg L=t,\ \prod_{y\in Y}(X-y)\mid L\Bigr\}
\]
and
\[
\mathcal L=\{L\in\f_{q^2}[X]\ \text{monic, self-dual},\ \deg L=t,\ \text{gcd}(L,X^{(q+1)/d}-1)=1\}.
\]
Then $|\mathcal L_Y|=\Lambda_{t-|Y|}$ (which is $0$ if $t-|Y|<0$). By inclusion-exclusion,
\begin{align*}
|\mathcal L|\,&=\sum_{i=0}^t(-1)^i\binom{(q+1)/d}i\Lambda_{t-i}\cr
&=\sum_{i=0}^{t-1}(-1)^i\binom{(q+1)/d}i(q+1)q^{t-i-1}+(-1)^t\binom{(q+1)/d}t.
\end{align*}
On the other hand, we have
\[
(q^2-1)|\mathcal L|=(q+1)l(t,0),
\]
since both sides count the number of self-dual polynomials of degree $t$ in $\f_{q^2}[X]$ that are relatively prime to $X^{(q+1)/d}-1$. Hence $l(t,0)=(q-1)|\mathcal L|$ and the conclusion follows.
\end{proof}

However, for $\tau>0$, we have not found an explicit formula for $l(t,\tau)$.

\begin{ques}\label{Q2.6}
For $(q+1)/d-t\le \tau\le t<(q+1)/d$, determine
\begin{align*}
l(t,\tau)=|\{&L=P+X^{(q+1)/d-\tau}Q: P,Q\in\f_{q^2}[X],\ \deg P=t-\tau,\cr
&\deg Q=\tau+t-(q+1)/d,\ \tilde P=P,\ \tilde Q=Q,\ \text{\rm gcd}(L,X^{(q+1)/d}-1)=1\}|.
\end{align*}
\end{ques}

\section{Permutation Binomials}

We follow the notation of Algorithm~\ref{Algo2.4}. Assume that the polynomial $h(X)$ resulting from Algorithm~\ref{Algo2.4} is a binomial, i.e., the matrix $[a_{ij}]$ has precisely two nonzero entries. Without loss of generality, assume that
\[
[a_{ij}]=
\bbordermatrix{ & \scriptstyle 0 && \scriptstyle v  \cr
              \scriptstyle 0 &1  \cr
              \cr
              \scriptstyle u & & & a & &\cr
              \cr},
\]
where $a\in\f_{q^2}^*$, $0\le u<(q+1)/d$, $0\le v<d$, $(u,v)\ne(0,0)$. We remind the reader that the rows of the matrix $[a_{ij}]$ are labeled by integers $0,\dots,(q+1)/d-1$ and the columns are labeled by $0,\dots,d-1$.

\medskip

{\bf Case 1.} Assume that $u=0$. Then 
\[
X^rh(X^{q-1})=X^r(1+aX^{v(q^2-1)/d}).
\]
It is well known, as stated in the introduction, that $X^r(1+aX^{r(q^2-1)/d})$ is a PP of $\f_{q^2}$ if and only if $\text{gcd}(r,(q^2-1)/d)=1$ and $X^r(1+aX^v)^{(q^2-1)/d}$ permutes $\mu_d$.

\medskip

{\bf Case 2.} Assume that $u>0$. Then
\vspace{-0.8em} 
\[
[M_{ik}]=
\bbordermatrix{{}\cr
\scriptstyle 0 & 1 & \cdots &1 \cr
\cr
\scriptstyle u & a\epsilon^{v\cdot 0} & \cdots & a\epsilon^{v(d-1)} \cr
\cr}
\]
and
\[
\sum_iM_{ik}X^i=1+a\epsilon^{vk}X^u,\quad s_k=0,\ t_k=u,\quad 0\le k<d.
\]
In particular, $L_0=1+aX^u\in\mathcal L_0(t_0,\tau_0;\lambda_0)$, where $\tau_0\in\{0\}\cup[(q+1)/d-t_0,t_0]$. 

First assume that $\tau_0=0$. By the definition of $\mathcal L_0(t_0,0;\lambda_0)$, $L_0$ is self-dual. It follows that $a\in\mu_{q+1}$. 
We have $X^rh(X^{q-1})=X^r(1+aX^{l(q-1)})$, where $l=u+v(q+1)/d$. Because of the condition $a\in\mu_{q+1}$, such permutation binomials are well known. By \cite[Corollary~5.3]{Zieve-arXiv1310.0776}, $X^r(1+aX^{l(q-1)})$ permutes $\f_{q^2}$ if and only if $\text{gcd}(r,q-1)=1$, $\text{gcd}(r-l,q+1)=1$ and $(-a)^{(q+1)/\text{gcd}(q+1,l)}\ne 1$.

Next, assume that $\tau_0\in[(q+1)/d-t_0,t_0]$. Since $L_0\in\mathcal L_0(t_0,\tau_0;\lambda_0)$, we have $L_0=P+X^{(q+1)-\tau_0}Q$, where $P,Q\in\f_{q^2}[X]$, $\deg P=t_0-\tau_0$, $\deg Q=\tau_0+t_0-(q+1)/d$, $\text{gcd}(L_0,X^{(q+1)/d}-1)=1$.  Since $P+X^{(q+1)-\tau_0}Q$ is a binomial, we must have $t_0=\tau_0$ and $(q+1)/d-\tau_0=t_0$ (see Figure~\ref{F3}). Hence $t_0=\tau_0=u=(q+1)/2d$. Then $h(X)=1+aX^{u+v(q+1)/d}=1+aX^{(1+2v)(q+1)/2d}$. Then $X^rh(X^{q-1})=X^r(1+aX^{(1+2v)(q^2-1)/2d})$ is a PP of $\f_{q^2}$ if and only if $\text{gcd}(r,(q^2-1)/2d)=1$ and $X^r(1+aX^{1+2v})$ permutes $\mu_{2d}$.

\begin{figure}
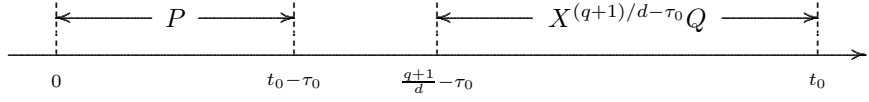

\[
\beginpicture
\setcoordinatesystem units <1.8em, 2em> point at 0 0

\arrow <5pt> [.2,.67] from -9 0 to 9 0 
\arrow <5pt> [.2,.67] from -6 0.7 to -8 0.7 
\arrow <5pt> [.2,.67] from -5 0.7 to -3 0.7
\arrow <5pt> [.2,.67] from 2 0.7 to 0 0.7
\arrow <5pt> [.2,.67] from 6 0.7 to 8 0.7

\setdashes<0.18em>
\plot -8 0   -8 1 /
\plot -3 0   -3 1 /
\plot  0 0   0 1 /
\plot 8 0   8 1 /

\put {$\scriptstyle 0$}  at -8 -0.5
\put {$\scriptstyle t_0-\tau_0$}  at -3 -0.5
\put {$\scriptstyle \frac{q+1}d-\tau_0$}  at 0 -0.5
\put {$\scriptstyle t_0$}  at 8 -0.5
\put {$P$} at -5.5 0.7
\put {$X^{(q+1)/d-\tau_0}Q$}  at 4 0.7

\endpicture
\]
\caption{When $P+X^{(q+1)/d-\tau_0}Q$ is a binomial}\label{F3}
\end{figure}

\medskip
{\bf Summary for binomials.} From the above two cases, we see that permutation binomials generated by Algorithm~\ref{Algo2.4} were all previously known.

\section{Permutation Trinomials}

Now assume that $h(X)$ in Algorithm~\ref{Algo2.4} is a trinomial, i.e., the matrix $[a_{ij}]$ has precisely three nonzero entries. Without loss of generality, write
\[
[a_{ij}]=
\bbordermatrix{&\scriptstyle 0 &&\scriptstyle  j_1 &&&\scriptstyle  j_2&\cr 
\scriptstyle 0 & 1\cr
\cr
\scriptstyle i_1 &&& a\cr
\cr
\scriptstyle i_2 &&&&&& b\cr
\cr
},\qquad a,b\in\f_{q^2}^*.
\]

\subsection{Three cases}\

\medskip

{\bf Case 1.} Assume that $i_1=i_2=0$. Then
\[
X^rh(X^{q-1})=X^r(1+aX^{j_1(q^2-1)/d}+bX^{j_2(q^2-1)/d}).
\]
Such a trinomial is a PP of $\f_{q^2}$ if and only if $\text{gcd}(r,(q^2-1)/d)=1$ and $X^r(1+aX^{j_1}+bX^{j_2})^{(q^2-1)/d}$ permutes $\mu_d$.

\medskip

{\bf Case 2.} Assume that $i_1=0$ and $0<i_2<(q+1)/d$. Then
\[
[M_{ik}]=\bbordermatrix{{}\cr
\scriptstyle 0 &1+a\epsilon^{j_1\cdot 0}& \cdots & 1+a\epsilon^{j_1(d-1)}\cr
\cr
\scriptstyle i_2 &b\epsilon^{j_2\cdot 0}&\cdots& b\epsilon^{j_2(d-1)}\cr
\cr}
\]
and
\begin{equation}\label{2.16.1}
\sum_i M_{ik}X^i=1+a\epsilon^{j_1k}+b\epsilon^{j_2k}X^{i_2}.
\end{equation}

\medskip
{\bf Case 2.1.} Assume that $1+a\epsilon^{j_1k}\ne 0$ for all $0\le k<d$. Then
\[
L_k(X)=1+a\epsilon^{j_1k}+b\epsilon^{j_2k}X^{i_2},\quad  s_k=0,\quad t_k=i_2.
\]

\begin{lem}\label{L4.1}
The following statements hold in Case 2.1.
\begin{itemize}
\item[(i)] If $\tau_k=0$ for some $0\le k<d$, then
\begin{equation}\label{2.17}
\Bigl(\frac{1+a^q\epsilon^{-j_1k}}b\Bigr)^{(q+1)/d}=\epsilon^{\alpha(k)}
\end{equation}
for some $\alpha(k)\in\Bbb Z/d\Bbb Z$. We have 
\[
e_k=r-i_2\quad\text{and}\quad \pi(k)=-j_2k\frac{q+1}d+\alpha(k).
\]

\item[(ii)]
If $\tau_k=0$ for all $0\le k<d$, then $j_1=d/2$, $a^{q-1}=-1$, $(1-a)/b\in\mu_{q+1}$, $e_k=r-i_2$, and
\[
\pi(k)+e_kk=\Bigl(-j_2\frac{q+1}d+r-i_2\Bigr)k+\delta(k)v+u,
\]
where 
\[
\Bigl(\frac{1-a}b\Bigr)^{(q+1)/d}=\epsilon^u,\quad \Bigl(\frac{1+a}{1-a}\Bigr)^{(q+1)/d}=\epsilon^v,
\]
and 
\[
\delta(k)=\begin{cases}
0&\text{if $k$ is even},\cr
1&\text{if $k$ is odd}.
\end{cases}
\]
\end{itemize}
\end{lem}

\begin{proof}
(i) Clearly, $e_k=r-i_2$. Since $L_k\in\mathcal L_k(t_k,0;\lambda_k)$, we have $\tilde L_k=\lambda_kL_k$, whence 
\[
\lambda_k=\frac{(1+a\epsilon^{j_1k})^q}{b\epsilon^{j_2k}}.
\]
Since 
\[
\epsilon^{\pi(k)}=\lambda_k^{(q+1)/d}=\Bigl(\frac{1+a^q\epsilon^{-j_1k}}{b\epsilon^{j_2k}}\Bigr)^{(q+1)/d}=\epsilon^{-j_2k(q+1)/d}\Bigl(\frac{1+ a^q\epsilon^{-j_1k}}b\Bigr)^{(q+1)/d},
\]
we have 
\[
\Bigl(\frac{1+a^q\epsilon^{-j_1k}}b\Bigr)^{(q+1)/d}=\epsilon^{\alpha(k)}
\]
for some $\alpha(k)\in\Bbb Z/d\Bbb Z$ and $\pi(k)=-j_2k(q+1)/d+\alpha(k)$.

\medskip
(ii) By \eqref{2.17},
\begin{equation}\label{2.18}
(1+a^q\epsilon^{-j_1k})^{q+1}=b^{q+1},
\end{equation}
i.e.,
\[
(1+a\epsilon^{j_1k})(1+a^q\epsilon^{-j_1k})=b^{q+1}.
\]
Hence the quadratic equation $(1+ax)(1+a^qx^{-1})=b^{q+1}$ has solutions $x=\epsilon^{j_1k}$, $0\le k<d$. Since the number of such solutions is $\le 2$ and since $0<j_1<d$, we must have $j_1=d/2$, whence $\epsilon^{j_1}=-1$. It follows from \eqref{2.18}, with $k=0,1$, that 
\[
(1+a^q)^{q+1}=(1-a^q)^{q+1}.
\]
This happens if and only if $a^q=-a$. (Note that $q$ is odd since $2\mid d$.) 
Then
\[
\Bigl(\frac{1-a}b\Bigr)^{q+1}=1\quad\text{and}\quad \Bigl(\frac{1+a}{1-a}\Bigr)^{q+1}=1.
\]
Write
\[
\Bigl(\frac{1-a}b\Bigr)^{(q+1)/d}=\epsilon^u\quad\text{and}\quad \Bigl(\frac{1+a}{1-a}\Bigr)^{(q+1)/d}=\epsilon^v.
\]
Then by \eqref{2.17},
\begin{align*}
\epsilon^{\alpha(k)}\,&=\Bigl(\frac{1-a(-1)^k}b\Bigr)^{(q+1)/d}=\Bigl(\frac{1-a}b\Bigr)^{(q+1)/d}\Bigl(\frac{1-a(-1)^k}{1-a}\Bigr)^{(q+1)/d}\cr
&=\begin{cases}
\epsilon^u&\text{if $k$ is even},\cr
\epsilon^{u+v}&\text{if $k$ is odd}.
\end{cases}
\end{align*}
Thus $\alpha(k)=u+\delta(k)v$, and by (i),
\[
\pi(k)+e_kk=-j_2k\frac{q+1}d+u+\delta(k)v+(r-i_2)k =\Bigl(-j_2\frac{q+1}d+r-i_2\Bigr)k+\delta(k)v+u.
\]
\end{proof}

If $\tau_k\in[(q+1)/d-t_k,t_k]$ for some $0\le k<d$, applying the argument in the last paragraph in Case 2 of Section 3 to $L_k\in\mathcal L_k(t_k,\tau_k;\lambda_k)$, we have $t_k=\tau_k=i_2=(q+1)/2d$. Therefore,
\begin{align*}
X^rh(X^{q-1})\,&=X^r(1+aX^{(q-1)\cdot j_1(q+1)/d}+bX^{(q-1)((q+1)/2d+j_2(q+1)/d)})\cr
&=X^r(1+aX^{j_1(q^2-1)/d}+bX^{(2j_2+1)(q^2-1)/2d)}).
\end{align*}
This trinomial permutes $\f_{q^2}$ if and only if $\text{gcd}(r,(q^2-1)/2d)=1$ and $X^r(1+aX^{2j_1}+bX^{2j_2+1})^{(q^2-1)/2d}$ permutes $\mu_{2d}$.

\medskip
{\bf Case 2.2.} Assume that $1+a\epsilon^{j_1k}=0$ for some $0\le k<d$. Write
$\epsilon^k=\lambda^{(q+1)/d}$ for some $\lambda\in\mu_{q+1}$. Then
\begin{align*}
h(\lambda X)\,&=1+a(\lambda X)^{j_1(q+1)/d}+b(\lambda X)^{i_2+j_2(q+1)/d}\cr
&=1-X^{j_1(q+1)/d}+b\lambda^{i_2}\epsilon^{j_2k}X^{i_2+j_2(q+1)/d}.
\end{align*}
Hence we may assume that $a=-1$. By \eqref{2.16.1},
\[
L_k(X)=\begin{cases}
b\epsilon^{j_2k}&\text{if}\ j_1k\equiv 0\pmod d,\cr
1-\epsilon^{j_1k}+b\epsilon^{j_2k}X^{i_2}&\text{if}\ j_1k\not\equiv 0\pmod d,
\end{cases}
\]
\begin{equation}\label{sk}
s_k=\begin{cases}
i_2&\text{if}\ j_1k\equiv 0\pmod d,\cr
0&\text{if}\ j_1k\not\equiv 0\pmod d,
\end{cases}
\end{equation}
\begin{equation}\label{tk}
t_k=\begin{cases}
0&\text{if}\ j_1k\equiv 0\pmod d,\cr
i_2&\text{if}\ j_1k\not\equiv 0\pmod d.
\end{cases}
\end{equation}

\begin{lem}\label{L4.2}
If $j_1k\equiv 0\pmod d$, then $\tau_k=0$, $e_k=r-2i_2$, and 
\[
\pi(k)=-2j_2k\frac{q+1}d+\beta, 
\]
where $b^{(q^2-1)/d}=\epsilon^\beta$.
\end{lem}

\begin{proof}
Clearly, $\tau_k=0$ and $e_k=r-2i_2$. Since $L_k=b\epsilon^{j_2k}\in\mathcal L_k(t_k,0;\lambda_k)$, we have $\tilde L_k=\lambda_kL_k$, whence $\lambda_k=(b\epsilon^{j_2k})^{q-1}=b^{q-1}\epsilon^{-2j_2k}$. Since
\[
\lambda_k^{(q+1)/d}=b^{(q^2-1)/d}\epsilon^{-2j_2k(q+1)/d}=\epsilon^{-2j_2k(q+1)/d+\beta},
\]
we have $\pi(k)=-2j_2k(q+1)/d+\beta$.
\end{proof}

If $\tau_k\in[(q+1)/d-t_k,t_k]$ for some $0\le k<d$ with $j_1k\not\equiv 0\pmod d$, applying the argument in the last paragraph in Case 2 of Section 3 to $L_k\in\mathcal L_k(t_k,\tau_k;\lambda_k)$, we have $t_k=\tau_k=i_2=(q+1)/2d$. Then
\[
X^rh(X^{q-1})=X^r(1+aX^{j_1(q^2-1)/d}+bX^{(2j_2+1)(q^2-1)/2d)}),
\]
which permutes $\f_{q^2}$ if and only if $\text{gcd}(r,(q^2-1)/2d)=1$ and $X^r(1+aX^{2j_1}+bX^{2j_2+1})^{(q^2-1)/2d}$ permutes $\mu_{2d}$. Therefore, we assume that $\tau_k=0$ for all $0\le k<d$ with $j_1k\not\equiv 0\pmod d$. This assumption combined with Lemma~\ref{L4.2} means that $\tau_k=0$ for all $0\le k<d$.

\begin{lem}\label{L4.3}
Assume that $\tau_k=0$ for all $0\le k<d$ in Case 2.2. Then $o(\epsilon^{j_1})=2$ or $3$.

\begin{itemize}
\item[(i)]
If $o(\epsilon^{j_1})=2$, then
\[
s_k=\begin{cases}
i_2&\text{if}\ k\equiv 0\pmod 2,\cr
0&\text{if}\ k\not\equiv 0\pmod 2,
\end{cases}
\]

\[
t_k=\begin{cases}
0&\text{if}\ k\equiv 0\pmod 2,\cr
i_2&\text{if}\ k\not\equiv 0\pmod 2,
\end{cases}
\]

\[
e_k=\begin{cases}
r-2i_2&\text{if}\ k\equiv 0\pmod 2,\cr
r-i_2&\text{if}\ k\not\equiv 0\pmod 2,
\end{cases}
\]

\[
\pi(k)=\begin{cases}
\displaystyle -2j_2k\frac{q+1}d+2\theta&\text{if}\ k\equiv 0\pmod 2,\vspace{0.4em}\cr
\displaystyle -j_2k\frac{q+1}d+\theta&\text{if}\ k\not\equiv 0\pmod 2,
\end{cases}
\]
where $(2/b)^{(q+1)/d}=\epsilon^\theta$.

\medskip

\item[(ii)] 
If $o(\epsilon^{j_1})=3$, then
\[
s_k=\begin{cases}
i_2&\text{if}\ k\equiv 0\pmod 3,\cr
0&\text{if}\ k\not\equiv 0\pmod 3,
\end{cases}
\]

\[
t_k=\begin{cases}
0&\text{if}\ k\equiv 0\pmod 3,\cr
i_2&\text{if}\ k\not\equiv 0\pmod 3,
\end{cases}
\]

\[
e_k=\begin{cases}
r-2i_2&\text{if}\ k\equiv 0\pmod 3,\cr
r-i_2&\text{if}\ k\not\equiv 0\pmod 3,
\end{cases}
\]

\[
\pi(k)=\begin{cases}
\displaystyle -(2j_2k+j_1)\frac{q+1}d+2\eta+\frac{q+1}{\text{\rm gcd}(2,d)} &\text{if}\ k\equiv 0\pmod 3,\vspace{0.4em}\cr
\displaystyle -(j_2k+j_1)\frac{q+1}d+\eta+\frac{q+1}{\text{\rm gcd}(2,d)}&\text{if}\ k\not\equiv 0\pmod 3,
\end{cases}
\]
where 
\[
\Bigl(\frac{1-\epsilon^{j_1}}b\Bigr)^{(q+1)/d}=\epsilon^\eta.
\]
\end{itemize}
\end{lem} 

\begin{proof}
If $j_1k\not\equiv 0\pmod d$, then $L_k=1-\epsilon^{j_1k}+b\epsilon^{j_2k}X^{i_2}\in\mathcal L_k(t_k,0,\lambda_k)$. Since $\tilde L_k=\lambda_kL_k$, we have 
\[
\frac{(1-\epsilon^{j_1k})^q}{b\epsilon^{j_2k}}=\lambda_k.
\]
It follows that
\[
1=\Bigl(\frac{(1-\epsilon^{j_1k})^q}{b\epsilon^{j_2k}}\Bigr)^{q+1}=\Bigl(\frac{1-\epsilon^{j_1k}}b\Bigr)^{q+1},
\]
i.e.,
\begin{equation}\label{2.20}
(1-\epsilon^{j_1k})(1-\epsilon^{-j_1k})=b^{q+1}.
\end{equation}
Therefore, $\epsilon^{j_1k}$ is a root of
\begin{equation}\label{2.21}
(1-x)(1-x^{-1})=b^{q+1}
\end{equation}
whenever $\epsilon^{j_1k}\ne 1$. Since \eqref{2.21} has at most two solutions, we have $o(\epsilon^{j_1})\le 3$.

\medskip
(i) Assume that $o(\epsilon^{j_1})=2$, i.e., $\epsilon^{j_1}=-1$. The formulas for $s_k$ and $t_k$ follow from \eqref{sk} and \eqref{tk}, and the formula for $e_k$ is obvious. It remains to prove the formula for $\pi(k)$. 

For $k\not\equiv 0\pmod 2$,
\[
\epsilon^{\pi(k)}=\lambda_k^{(q+1)/d}=\Bigl(\frac 2{b\epsilon^{j_2k}}\Bigr)^{(q+1)/d}=\Bigl(\frac 2b\Bigr)^{(q+1)/d}\epsilon^{-j_2k(q+1)/d}=\epsilon^{-j_2k(q+1)/d+\theta},
\]
where $(2/b)^{(q+1)/d}=\epsilon^\theta$. Hence 
\[
\pi(k)=-j_2k\frac{q+1}d+\theta.
\]
For $k\equiv 0\pmod 2$, by Lemma~\ref{L4.2}
\[
\pi(k)=-2j_2k\frac{q+1}d+\beta,
\]
where $b^{(q^2-1)/d}=\epsilon^\beta$. Since 
\[
\epsilon^{-2\theta}=\epsilon^{\theta(q-1)}=b^{-(q^2-1)/d}=\epsilon^{-\beta},
\]
we have $\beta=2\theta$.

\medskip
(ii) Assume that $o(\epsilon^{j_1})=3$ and write $\epsilon^{j_1}=\omega$. Again, we only have to prove the formula for $\pi(k)$.

For $k\not\equiv 0\pmod 3$,
\begin{align*}
\epsilon^{\pi(k)}\,&=\Bigl(\frac{1-\omega^k}{b\epsilon^{j_2k}}\Bigr)^{(q+1)/d}=\Bigl(\frac{1-\omega^k}{1-\omega}\Bigr)^{(q+1)/d}\Bigl(\frac{1-\omega}b\Bigr)^{(q+1)/d}\epsilon^{-j_2k(q+1)/d}\cr
&=\Bigl(\frac{1-\omega^k}{1-\omega}\Bigr)^{(q+1)/d}\epsilon^{-j_2k(q+1)/d+\eta},
\end{align*}
where $((1-\omega)/b)^{(q+1)/d}=\epsilon^\eta$. Note that 
\[
\Bigl(\frac{1-\omega^{-1}}{1-\omega}\Bigr)^{(q+1)/d}=(-\omega^{-1})^{(q+1)/d}=(-1)^{(q+1)/d}\epsilon^{-j_1(q+1)/d},
\]
where
\[
(-1)^{(q+1)/d}=\left.\begin{cases}
\epsilon^{(q+1)/2}&\text{if $d$ is even}\cr
1&\text{if\ $d$ is odd}
\end{cases}\right\}=\epsilon^{(q+1)/\text{gcd}(2,d)}.
\]
Hence 
\[
\pi(k)=-(j_2k+j_1)\frac{q+1}d+\eta+\frac{q+1}{\text{gcd}(2,d)}.
\]
For $k\equiv 0\pmod 3$, by Lemma~\ref{L4.2}, 
\[
\pi(k)=-2j_2k\frac{q+1}d+\beta,
\]
where $b^{(q^2-1)/d}=\epsilon^\beta$. Since 
\begin{align*}
\epsilon^{-2\eta}\,&=\epsilon^{\eta(q-1)}=\Bigl(\frac{1-\omega}b\Bigr)^{(q^2-1)/d}=\epsilon^{-\beta}\Bigl(\frac{1-\omega^{-1}}{1-\omega}\Bigr)^{(q+1)/d}\cr
&=\epsilon^{-\beta-j_1(q+1)/d}(-1)^{(q+1)/d}=\epsilon^{-\beta-j_1(q+1)/d+(q+1)/\text{gcd}(2,d)},
\end{align*}
we have 
\[
\beta=-j_1\frac{q+1}d+2\eta+\frac{q+1}{\text{gcd}(2,d)}.
\]
\end{proof}

\medskip
{\bf Case 3.} Assume that $0<i_1<i_2<(q+1)/d$. Then
\[
[M_{ik}]=\bbordermatrix{{}\cr
\scriptstyle 0& 1&\cdots&1\cr
\cr
\scriptstyle i_1& a\epsilon^{j_1\cdot 0}&\cdots&a\epsilon^{j_1(d-1)}\cr
\cr
\scriptstyle i_2& b\epsilon^{j_2\cdot 0}&\cdots&b\epsilon^{j_2(d-1)}\cr
\cr},
\]
\[
L_k(X)=1+a\epsilon^{j_1k}X^{i_1}+b\epsilon^{j_2k}X^{i_2},
\]
\[
s_k=0,\quad t_k=i_2.
\]

\begin{lem}\label{L4.4}
If $\tau_k=0$, then $i_1=i_2/2$, $e_k=r-i_2$, and 
\[
\pi(k)=-2j_1k\frac{q+1}d+\alpha,
\]
where $a^{(q^2-1)/d}=\epsilon^\alpha$ and $b=a^{1-q}\epsilon^{(2j_1-j_2)k}$.
\end{lem}

\begin{proof}
Clearly, $e_k=r-i_2$. Since $\tilde L_k=\lambda_kL_k$, i.e.,
\[
\overline{b\epsilon^{j_2k}}+\overline{a\epsilon^{j_1k}}X^{i_2-i_1}+X^{i_2}=\lambda_k(1+a\epsilon^{j_1k}X^{i_1}+b\epsilon^{j_2k}X^{i_2}),
\]
we have $i_1=i_2/2$ and $(\overline{b\epsilon^{j_2k}},\overline{a\epsilon^{j_1k}},1)=\lambda_k(1,a\epsilon^{j_1k},b\epsilon^{j_2k})$. Hence
\[
\lambda_k=\overline{b\epsilon^{j_2k}}=b^q\epsilon^{-j_2k}
\]
and
\[
\overline{a\epsilon^{j_1k}}\cdot b\epsilon^{j_2k}=a\epsilon^{j_1k},\quad\text{i.e.,}\quad b=a^{1-q}\epsilon^{(2j_1-j_2)k}.
\]
We have
\begin{align*}
\epsilon^{\pi(k)}\,&=\lambda_k^{(q+1)/d}=(b^q\epsilon^{-j_2k})^{(q+1)/d}=((a^{q-1}\epsilon^{-(2j_1-j_2)k}\cdot\epsilon^{-j_2k})^{(q+1)/d}\cr
&=a^{(q^2-1)/d}\epsilon^{-2j_1k(q+1)/d}=\epsilon^{-2j_1k(q+1)/d+\alpha},
\end{align*}
where $a^{(q^2-1)/d}=\epsilon^\alpha$. Hence
\[
\pi(k)=-2j_1k\frac{q+1}d+\alpha.
\]
\end{proof}

\noindent{\bf Remark.} 
If $\tau_k=0$ for all $0\le k<d$, then by Lemma~\ref{L4.4}, $i_1=i_2/2$, $2j_1-j_2\equiv 0\pmod d$ and $b=a^{1-q}$. Consequently,
\begin{align*}
h(X)\,&=1+aX^{i_1+j_1(q+1)/d}+bX^{i_2+j_2(q+1)/d}\cr
&\equiv 1+aX^{i_1+j_1(q+1)/d}+bX^{2(i_1+j_1(q+1)/d)}\pmod{X^{q+1}-1}\cr
&=h_1(X),
\end{align*}
where
\[
h_1(X)=1+aX^l+bX^{2l},\quad l=i_1+j_1(q+1)/d.
\]
Hence
\[
X^rh(X^{q-1})\equiv X^rh_1(X^{q-1})\pmod{X^{q^2-1}-1}.
\]
Since $b=a^{1-q}$, $h_1(X)$ is self-dual. In general, when $h_1(X)$ is self dual, PPs of $\f_{q^2}$ of the form $X^rh_1(X^{q-1})$ are known; see Example~\ref{E5.1}.

\begin{lem}\label{L4.5}
If $\tau_k\in[(q+1)/d-t_k,t_k]$, then precisely one of the following occurs.
\begin{itemize}
\item[(i)]
\[
t_k=\tau_k=\frac12\Bigl(\frac{q+1}d+1\Bigr),\quad i_1=\frac12\Bigl(\frac{q+1}d-1\Bigr),\quad i_2=\frac12\Bigl(\frac{q+1}d+1\Bigr),
\]
\[
a=b^q\epsilon^{-(j_1+j_2+1)k},\quad e_k=r,\quad \pi(k)=0.
\]

\item[(ii)]
\[
t_k=\frac12\Bigl(\frac{q+1}d+1\Bigr),\quad \tau_k=\frac12\Bigl(\frac{q+1}d-1\Bigr),\quad i_1=1,\quad i_2=\frac12\Bigl(\frac{q+1}d+1\Bigr),
\]
\[
a=b^{1-q}\epsilon^{(2j_2-j_1+1)k},\quad b^{(q^2-1)/d}=\epsilon^\beta,\quad e_k=r-1,
\]
\[
\pi(k)=-(2j_2+1)k\frac{q+1}d+\beta.
\]
\end{itemize}
\end{lem}

\begin{proof}
Since 
\[
L_k=1+a\epsilon^{j_1k}X^{i_1}+b\epsilon^{j_2k}X^{i_2}=P+X^{(q+1)/d-\tau_k}Q
\]
is a trinomial, where $\deg P=t_k-\tau_k$ and $\deg Q=\tau_k+t_k-(q+1)/d$ (see Figure~\ref{F4}), we have 
\[
\begin{cases}
t_k-\tau_k=0,\vspace{0.3em}\cr
\displaystyle t_k=\frac{q+1}d-\tau_k+1,\vspace{0.3em}\cr
\displaystyle i_1=\frac{q+1}d-\tau_k,\vspace{0.3em}\cr
i_2=t_k,
\end{cases}
\qquad\text{or}\qquad
\begin{cases}
t_k-\tau_k=1,\vspace{0.3em}\cr
\displaystyle t_k=\frac{q+1}d-\tau_k,\vspace{0.3em}\cr
i_1=1,\cr
i_2=t_k,
\end{cases}
\]
i.e.,
\begin{equation}\label{2.22}
\begin{cases}
\displaystyle t_k=\tau_k=\frac12\Bigl(\frac{q+1}d+1\Bigr),\vspace{0.3em}\cr
\displaystyle i_1=\frac12\Bigl(\frac{q+1}d-1\Bigr),\vspace{0.3em}\cr
\displaystyle i_2=\frac12\Bigl(\frac{q+1}d+1\Bigr),\cr
\end{cases}
\end{equation}
or
\begin{equation}\label{2.23}
\begin{cases}
\displaystyle t_k=\frac12\Bigl(\frac{q+1}d+1\Bigr),\vspace{0.3em}\cr
\displaystyle \tau_k=\frac12\Bigl(\frac{q+1}d-1\Bigr),\vspace{0.3em}\cr
i_1=1,\vspace{0.3em}\cr
\displaystyle i_2=\frac12\Bigl(\frac{q+1}d+1\Bigr).
\end{cases}
\end{equation}

\begin{figure}
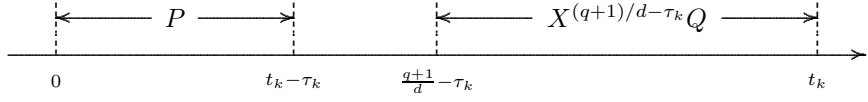

\[
\beginpicture
\setcoordinatesystem units <1.8em, 2em> point at 0 0

\arrow <5pt> [.2,.67] from -9 0 to 9 0 
\arrow <5pt> [.2,.67] from -6 0.7 to -8 0.7 
\arrow <5pt> [.2,.67] from -5 0.7 to -3 0.7
\arrow <5pt> [.2,.67] from 2 0.7 to 0 0.7
\arrow <5pt> [.2,.67] from 6 0.7 to 8 0.7

\setdashes<0.18em>
\plot -8 0   -8 1 /
\plot -3 0   -3 1 /
\plot  0 0   0 1 /
\plot 8 0   8 1 /

\put {$\scriptstyle 0$}  at -8 -0.5
\put {$\scriptstyle t_k-\tau_k$}  at -3 -0.5
\put {$\scriptstyle \frac{q+1}d-\tau_k$}  at 0 -0.5
\put {$\scriptstyle t_k$}  at 8 -0.5
\put {$P$} at -5.5 0.7
\put {$X^{(q+1)/d-\tau_k}Q$}  at 4 0.7

\endpicture
\]
\caption{When $P+X^{(q+1)/d-\tau_k}Q$ is a trinomial}\label{F4}
\end{figure}

\medskip

(i) Assume \eqref{2.22}. First, note that $e_k=r-t_k+\tau_k=r$. In this case,
\[
L_k=1+X^{\frac12(\frac{q+1}d-1)}(a\epsilon^{j_1k}+b\epsilon^{j_2k}X),
\]
where $P=1$ and $Q=a\epsilon^{j_1k}+b\epsilon^{j_2k}X$. We have $\lambda_k=\tilde P/P=1$ and 
\[
\lambda_k\epsilon^k=\frac{\tilde Q}Q=\frac{\overline{b\epsilon^{j_2k}}+\overline{a\epsilon^{j_1k}}X}{a\epsilon^{j_1k}+b\epsilon^{j_2k}X}.
\]
Hence
\[
\frac{\overline{b\epsilon^{j_2k}}}{a\epsilon^{j_1k}}=\epsilon^k,
\]
i.e., $a=b^q\epsilon^{-(j_1+j_2+1)k}$. Clearly, $\pi(k)=0$.

\medskip
(ii) Assume \eqref{2.23}. First, note that $e_k=r-t_k+\tau_k=r-1$. In this case,
\[
L_k=1+a\epsilon^{j_1k}X+X^{\frac 12(\frac{q+1}d+1)}\cdot b\epsilon^{j_2k},
\]
where 
$P=1+a\epsilon^{j_1k}X$ and $Q= b\epsilon^{j_2k}$. Since $\tilde P=\lambda_k P$ and $\tilde Q=\lambda_k\epsilon^k Q$, we have
\[
(a\epsilon^{j_1k})^q=\lambda_k\qquad\text{and}\qquad (b\epsilon^{j_2k})^{q-1}=\lambda_k\epsilon^k.
\]
It follows that
\[
a=\lambda_k^{-1}\epsilon^{-j_1k}=\epsilon^k(b\epsilon^{j_2k})^{1-q}\cdot\epsilon^{-j_1k}=b^{1-q}\epsilon^{(2j_2-j_1+1)k}.
\]
Since
\begin{align*}
\lambda_k^{(q+1)/d}\,&=(a^q\epsilon^{-j_1k})^{(q+1)/d}=(b^{q-1}\epsilon^{-(2j_2-j_1+1)k}\epsilon^{-j_1k})^{(q+1)/d}\cr
&=b^{(q^2-1)/d}\epsilon^{-(2j_2+1)k(q+1)/d}=\epsilon^{-(2j_2+1)k(q+1)/d+\beta},
\end{align*}
where $b^{(q^2-1)/d}=\epsilon^\beta$, we have
\[
\pi(k)=-(2j_2+1)k\frac{q+1}d+\beta.
\]
\end{proof}

\noindent{\bf Remark.}
In Lemma~\ref{L4.5}, if (i) occurs for all $0\le k<d$, then $j_1+j_2+1\equiv 0\pmod d$ and $a=b^q$. We have 
\begin{align*}
h(X)\,&=1+aX^{\frac 12(\frac{q+1}d-1)+j_1\frac{q+1}d}+bX^{\frac 12(\frac{q+1}d+1)+j_2\frac{q+1}d} \cr
&\equiv 1+aX^{\frac 12(\frac{q+1}d-1)+j_1\frac{q+1}d}+bX^{q+1-(\frac 12(\frac{q+1}d-1)+j_1\frac{q+1}d)} \pmod{X^{q+1}-1}\cr
&\equiv X^{q+1-l}h_1(X)\pmod{X^{q+1}-1},
\end{align*}
where
\[
l=\frac 12\Bigl(\frac{q+1}d-1\Bigr)+j_1\frac{q+1}d
\]
and 
\[
h_1(X)=b+X^l+aX^{2l}.
\]
Since $a=b^q$, $h_1(X)$ is self-dual. We have
\[
X^rh(X^{q-1})\equiv X^{r+(q+1-l)(q-1)}h_1(X^{q-1}) \pmod{X^{q^2-1}-1},
\]
and PPs of $\f_{q^2}$ of this type are known (Example~\ref{E5.1}).

Similarly, if (ii) occurs for all $0\le k<d$ in Lemma~\ref{L4.5}, then $2j_2-j_1+1\equiv 0\pmod d$ and $a=b^{1-q}$. We have
\begin{align*}
h(X)\,&=1+aX^{1+j_1\frac{q+1}d}+bX^{\frac 12(\frac{q+1}d+1)+j_2\frac{q+1}d} \cr
&\equiv 1+aX^{2(\frac 12(\frac{q+1}d+1)+j_2\frac{q+1}d)}+bX^{\frac 12(\frac{q+1}d+1)+j_2\frac{q+1}d} \pmod{X^{q+1}-1}\cr
&=h_1(X),
\end{align*}
where
\[
h_1(X)=1+bX^l+aX^{2l},\quad l=\frac 12\Bigl(\frac{q+1}d+1\Bigr)+j_2\frac{q+1}d.
\]
Since $a=b^{1-q}$, $h_1(X)$ is self-dual. We have
\[
X^rh(X^{q-1})\equiv X^rh_1(X^{q-1}) \pmod{X^{q^2-1}-1},
\]
and PPs of $\f_{q^2}$ of this type are known.

\begin{lem}\label{4.6}
If $(q+1)/d\ne 3$, then either Lemma~\ref{L4.4} occurs for $0\le k<d$, or  Lemma~\ref{L4.5} (i) occurs for $0\le k<d$, or  Lemma~\ref{L4.5} (ii) occurs for all $0\le k<d$.
\end{lem}

\begin{proof}
In Lemma~\ref{L4.4}, Lemma~\ref{L4.5} (i) and Lemma~\ref{L4.5} (ii), we have
\[
i_1=\frac{i_2}2,\quad 
\begin{cases}
i_1=\displaystyle\frac 12\Bigl(\frac{q+1}d-1\Bigr),\vspace{0.4em}\cr
i_2=\displaystyle\frac 12\Bigl(\frac{q+1}d+1\Bigr),
\end{cases}\ \text{and}\quad
\begin{cases}
i_1=1\vspace{0.4em}\cr
i_2=\displaystyle\frac 12\Bigl(\frac{q+1}d+1\Bigr),
\end{cases}
\]
respectively. Any two of these three conditions imply that $(q+1)/d=3$.
\end{proof}

Partition $\Bbb Z/d\Bbb Z$ as 
\begin{equation}\label{K012}
\Bbb Z/d\Bbb Z=K_0\sqcup K_1\sqcup K_2,
\end{equation}
where Lemma~\ref{L4.4} occurs for $k\in K_0$, Lemma~\ref{L4.5} (i) occurs for all $k\in K_1$, and  Lemma~\ref{L4.5} (ii) occurs for all $k\in K_2$. Assume that at least two of $K_0,K_1,K_2$ are nonempty. Then $(q+1)/d=3$, $i_1=1$ and $i_2=2$. Since $\text{gcd}(e_k,3)=\text{gcd}(e_k,(q+1)/d)=1$, it follows from Lemmas~\ref{L4.4} and \ref{L4.5} that one of $K_0,K_1,K_2$ must be empty:
\[
\begin{cases}
\text{if $r\equiv 0\pmod 3$, then $K_1=\emptyset$};\cr
\text{if $r\equiv 1\pmod 3$, then $K_2=\emptyset$};\cr
\text{if $r\equiv -1\pmod 3$, then $K_0=\emptyset$}.
\end{cases}
\]
It also follows from Lemmas~\ref{L4.4} and \ref{L4.5} that
\begin{equation}\label{ab}
\begin{cases}
\text{if $k_0\in K_0$, then $b=a^{1-q}\epsilon^{(2j_1-j_2)k_0}$};\cr
\text{if $k_1\in K_1$, then $a=b^q\epsilon^{-(j_1+j_2+1)k_1}$};\cr
\text{if $k_2\in K_2$, then $a=b^{1-q}\epsilon^{(2j_2-j_1+1)k_2}$}.
\end{cases}
\end{equation}
Any combination of two of the above three equations allow us to determine $a$ and $b$ up to a third root of unity.

\begin{lem}\label{L4.7}
We have the following equivalences:
\begin{equation}\label{Claim4-1}
\begin{cases}
b=a^{1-q}\epsilon^{(2j_1-j_2)k_0}\cr
a=b^q\epsilon^{-(j_1+j_2+1)k_1}
\end{cases}
\Leftrightarrow\quad
\begin{cases}
a^3=\epsilon^{-(2j_1-j_2)k_0-(j_1+j_2+1)k_1}\cr
b=a^2\epsilon^{(2j_1-j_2)k_0}
\end{cases}
\end{equation}
\begin{equation}\label{Claim4-2}
\begin{cases}
b=a^{1-q}\epsilon^{(2j_1-j_2)k_0}\cr
a=b^{1-q}\epsilon^{(2j_2-j_1+1)k_2}
\end{cases}
\Leftrightarrow\quad
\begin{cases}
a^3=\epsilon^{-2(2j_1-j_2)k_0-(2j_2-j_1+1)k_2}\cr
b=a^2\epsilon^{(2j_1-j_2)k_0}
\end{cases}
\end{equation}
\begin{equation}\label{Claim4-3}
\begin{cases}
a=b^q\epsilon^{-(j_1+j_2+1)k_1}\cr
a=b^{1-q}\epsilon^{(2j_2-j_1+1)k_2}
\end{cases}
\Leftrightarrow\quad
\begin{cases}
b^3=\epsilon^{-(j_1+j_2+1)k_1-(2j_2-j_1+1)k_2}\cr
a=b^{-1}\epsilon^{-(j_1+j_2+1)k_1}
\end{cases}
\end{equation}
\end{lem}
 
\begin{proof}
We only prove \eqref{Claim4-1}. (The proofs of \eqref{Claim4-2} and \eqref{Claim4-3} are similar.)

\medskip

($\Rightarrow$) We have
\[
b^{q+1}=(a^{1-q}\epsilon^{(2j_1-j_2)k_0})^{q+1}=1,
\]
and hence
\[
a^{q+1}=(b^q\epsilon^{-(j_1+j_2+1)k_1})^{q+1}=1.
\]
Now 
\[
b=a^{1-q}\epsilon^{(2j_1-j_2)k_0}=a^2\epsilon^{(2j_1-j_2)k_0},
\]
and
\[
a=b^q\epsilon^{-(j_1+j_2+1)k_1}=(a^2\epsilon^{(2j_1-j_2)k_0})^q\epsilon^{-(j_1+j_2+1)k_1}=a^{-2}\epsilon^{-(2j_1-j_2)k_0-(j_1+j_2+1)k_1},
\]
i.e.,
\[
a^3=\epsilon^{-(2j_1-j_2)k_0-(j_1+j_2+1)k_1}.
\]

\medskip
($\Leftarrow$) First, since $(q+1)/d=3$, it is clear that $a^{q+1}=b^{q+1}=1$. The rest is obvious.
\end{proof}

\begin{lem}\label{L4.8}
Assume that $K_0\ne\emptyset$ and $K_1\ne\emptyset$ in \eqref{K012} and hence $(q+1)/d=3$, $i_1=1$, $i_2=2$, and $r\equiv 1\pmod 3$. Then $d$ is even, and $K_0$ and $K_1$ form the two cosets of $2\Bbb Z/d\Bbb Z$ in $\Bbb Z/d\Bbb Z$. Moreover, $a^3=-1$, $b=\pm a^2$,
\begin{equation}\label{eq:L4.8}
\begin{cases}
\displaystyle 2j_1-j_2\equiv \frac d2\pmod d,\vspace{0.4em}\cr
\displaystyle j_1+j_2+1\equiv \frac d2\pmod d,
\end{cases}
\end{equation}
and $\pi(k)+e_kk=rk$ for $k\in\Bbb Z/d\Bbb Z$. More precisely, either 
\[
q\equiv 11\pmod{18}\quad\text{and}\quad h(X)=1+aX^{(q+1)/3}+bX^{(q+1)/6}
\]
or
\[
q\equiv 5\pmod{18}\quad\text{and}\quad h(X)=1+aX^{2(q+1)/3}+bX^{5(q+1)/6}.
\]
\end{lem}

\begin{proof}
If $2j_1-j_2\equiv 0\pmod d$, then by \eqref{ab}, $b=a^{1-q}$. Moreover, by the proof of Lemma~\ref{L4.4}, for all $k\in\Bbb Z/d\Bbb Z$, $L_k\in\mathcal L_k(2,0;\lambda_k)$ with $\lambda_k=b^q\epsilon^{-j_2k}$. This means that Lemma~\ref{L4.4} occurs for all $k\in\Bbb Z/d\Bbb Z$, i.e., $K_0=\Bbb Z/d\Bbb Z$, which is a contradiction since $K_1\ne\emptyset$. Similarly, if $j_1+j_2+1\equiv 0\pmod d$, then by \eqref{ab}, $a=b^q$. Moreover, by the proof of Lemma~\ref{L4.5} (i), for all $k\in\Bbb Z/d\Bbb Z$, $L_k\in\mathcal L_k(2,2;1)$. This means that Lemma~\ref{L4.5} (i) occurs for all $k\in\Bbb Z/d\Bbb Z$, i.e., $K_1=\Bbb Z/d\Bbb Z$, which is also a contradiction. Hence $2j_1-j_2\not\equiv 0\pmod d$ and $j_1+j_2+1\not\equiv 0\pmod d$.

By \eqref{ab}, there exist $u,v\in\Bbb Z/d\Bbb Z$ such that
\[
\begin{cases}
(2j_1-j_2)k\equiv u\pmod d\quad \text{for all}\ k\in K_0,\cr
(j_1+j_2+1)k\equiv v\pmod d\quad \text{for all}\ k\in K_1.
\end{cases}
\]
Since $2j_1-j_2\not\equiv 0\pmod d$ and $j_1+j_2+1\not\equiv 0\pmod d$ and $K_0\cup K_1=\Bbb Z/d\Bbb Z$, we must have
\[
\begin{cases}
\displaystyle 2j_1-j_2\equiv\frac d2\pmod d,\vspace{0.4em}\cr
\displaystyle j_1+j_2+1\equiv \frac d2\pmod d,
\end{cases}
\]
$\{u,v\}=\{0,d/2\}$, and
\begin{gather*}
K_0=\{k\in\Bbb Z/d\Bbb Z:(2j_1-j_2)k\equiv u\pmod d\},\cr
K_1=\{k\in\Bbb Z/d\Bbb Z:(j_1+j_2+1)k\equiv v\pmod d\}.
\end{gather*}
It follows from \eqref{Claim4-1} that $a^3=\epsilon^{-u-v}=\epsilon^{d/2}=-1$ and $b=a^2\epsilon^u=\pm a^2$.

By Lemma~\ref{L4.4} and Lemma~\ref{L4.5} (i), we have
\[
\pi(k)+e_kk=\begin{cases}
(r-2-6j_1)k&\text{if}\ k\in K_0,\cr
rk&\text{if}\ k\in K_1.
\end{cases}
\]
By \eqref{eq:L4.8}, $3j_1+1\equiv 0\pmod d$, hence $\pi(k)+e_kk=rk$ for $k\in\Bbb Z/d\Bbb Z$.

System \eqref{eq:L4.8} is equivalent to
\[
\begin{cases}
1+3j_1\equiv 0\pmod d,\vspace{0.3em}\cr
\displaystyle 2+3j_2\equiv\frac d2\pmod d.
\end{cases}
\]
Since $0\le j_1,j_2<d$, we have
\[
\begin{cases}
1+3j_1=d,\vspace{0.3em}\cr
\displaystyle 2+3j_2=\frac d2,
\end{cases}
\quad\text{or}\quad
\begin{cases}
1+3j_1=2d,\vspace{0.3em}\cr
\displaystyle 2+3j_2=\frac{5d}2.
\end{cases}
\]
In the first case, $d\equiv 4\pmod 6$, whence $q=3d-1\equiv 11\pmod{18}$, and
\[
h(X)=1+aX^{1+3j_1}+bX^{2+3j_2}=1+aX^{(q+1)/3}+bX^{(q+1)/6}.
\]
In the second case, $d\equiv 2\pmod 6$, whence $q=3d-1\equiv 5\pmod{18}$, and
\[
h(X)=1+aX^{1+3j_1}+bX^{2+3j_2}=1+aX^{2(q+1)/3}+bX^{5(q+1)/6}.
\] 
\end{proof}

\begin{lem}\label{L4.9}
Assume that $K_0\ne\emptyset$ and $K_2\ne\emptyset$ in \eqref{K012} and hence $(q+1)/d=3$, $i_1=1$, $i_2=2$, and $r\equiv 0\pmod 3$. Then $d$ is even, and $K_0$ and $K_2$ form the two cosets of $2\Bbb Z/d\Bbb Z$ in $\Bbb Z/d\Bbb Z$. Moreover, $a^3=\pm1$, $b=-a^{-1}$,
\begin{equation}\label{class8-1}
\begin{cases}
\displaystyle 2j_1-j_2\equiv \frac d2\pmod d,\vspace{0.4em}\cr
\displaystyle 2j_2-j_1+1\equiv \frac d2\pmod d,
\end{cases}
\end{equation}
and $\pi(k)+e_kk=rk$ for $k\in\Bbb Z/d\Bbb Z$. More precisely, either 
\[
q\equiv 5\pmod{18}\quad\text{and}\quad h(X)=1+aX^{(q+1)/6}+bX^{5(q+1)/6}
\]
or
\[
q\equiv 11\pmod{18}\quad\text{and}\quad h(X)=1+aX^{5(q+1)/6}+bX^{(q+1)/6}.
\]
\end{lem}

\begin{proof}
By the same argument in the proof of Lemma~\ref{L4.8}, we have 
\[
\begin{cases}
\displaystyle 2j_1-j_2\equiv\frac d2\pmod d,\vspace{0.4em}\cr
\displaystyle 2j_2-j_1+1\equiv \frac d2\pmod d,
\end{cases}
\]
and 
\begin{gather*}
K_0=\{k\in\Bbb Z/d\Bbb Z:(2j_1-j_2)k\equiv u\pmod d\},\cr
K_2=\{k\in\Bbb Z/d\Bbb Z:(2j_2-j_1+1)k\equiv v\pmod d\},
\end{gather*}
where $\{u,v\}=\{0,d/2\}$. It follows from \eqref{Claim4-2} that $a^3=\epsilon^v=\pm 1$ and $b=a^2\epsilon^u=-a^2\epsilon^v=-a^5=-a^{-1}$. By Lemma~\ref{L4.4} and Lemma~\ref{L4.5} (ii), we have
\[
\pi(k)+e_kk=\begin{cases}
(r-2-6j_1)k&\text{if}\ k\in K_0,\cr
(r-4-6j_2)k&\text{if}\ k\in K_2.
\end{cases}
\]
By \eqref{class8-1}, $6j_1\equiv -2\pmod d$ and $6j_2\equiv -4\pmod d$, hence $\pi(k)+e_kk=rk$ for $k\in\Bbb Z/d\Bbb Z$.

System \eqref{class8-1} is equivalent to
\[
\begin{cases}
\displaystyle  1+3j_1\equiv \frac d2\pmod d,\vspace{0.3em}\cr
\displaystyle 2+3j_2\equiv\frac d2\pmod d,
\end{cases}
\]
i.e.,
\[
\begin{cases}
\displaystyle 1+3j_1=\frac d2,\vspace{0.3em}\cr
\displaystyle 2+3j_2=\frac {5d}2,
\end{cases}
\quad\text{or}\quad
\begin{cases}
\displaystyle 1+3j_1=\frac{5d}2,\vspace{0.3em}\cr
\displaystyle 2+3j_2=\frac d2.
\end{cases}
\]
In the first case, $d\equiv 2\pmod 6$, whence $q=3d-1\equiv 5\pmod{18}$, and
\[
h(X)=1+aX^{(q+1)/6}+bX^{5(q+1)/6}.
\]
In the second case, $d\equiv 4\pmod 6$, whence $q=3d-1\equiv 11\pmod{18}$, and
\[
h(X)=1+aX^{5(q+1)/6}+bX^{(q+1)/6}.
\] 
\end{proof}

\begin{lem}\label{L4.10}
Assume that $K_1\ne\emptyset$ and $K_2\ne\emptyset$ in \eqref{K012} and hence $(q+1)/d=3$, $i_1=1$, $i_2=2$, and $r\equiv -1\pmod 3$. Then $d$ is even, and $K_1$ and $K_2$ form the two cosets of $2\Bbb Z/d\Bbb Z$ in $\Bbb Z/d\Bbb Z$. Moreover, $b^3=-1$, $a=\pm b^{-1}$,
\begin{equation}\label{class9-1}
\begin{cases}
\displaystyle j_1+j_2+1\equiv \frac d2\pmod d,\vspace{0.4em}\cr
\displaystyle 2j_2-j_1+1\equiv \frac d2\pmod d,
\end{cases}
\end{equation}
and $\pi(k)+e_kk=rk$ for $k\in\Bbb Z/d\Bbb Z$. More precisely, either 
\[
q\equiv 5\pmod{18}\quad\text{and}\quad h(X)=1+aX^{(q+1)/6}+bX^{(q+1)/3}
\]
or
\[
q\equiv 11\pmod{18}\quad\text{and}\quad h(X)=1+aX^{5(q+1)/6}+bX^{2(q+1)/3}.
\]
\end{lem}

\begin{proof}
Again, by the argument in the proof of Lemma~\ref{L4.8}, we have 
\[
\begin{cases}
\displaystyle j_1+j_2+1\equiv \frac d2\pmod d,\vspace{0.4em}\cr
\displaystyle 2j_2-j_1+1\equiv \frac d2\pmod d,
\end{cases}
\]
and 
\begin{gather*}
K_1=\{k\in\Bbb Z/d\Bbb Z:(j_1+j_2+1)k\equiv u\pmod d\},\cr
K_2=\{k\in\Bbb Z/d\Bbb Z:(2j_2-j_1+1)k\equiv v\pmod d\},
\end{gather*}
where $\{u,v\}=\{0,d/2\}$. It follows from \eqref{Claim4-3} that $b^3=\epsilon^{u+v}=-1$ and $a=b^{-1}\epsilon^u=\pm b^{-1}$. By  Lemma~\ref{L4.5} (i) and (ii), we have
\[
\pi(k)+e_kk=\begin{cases}
rk&\text{if}\ k\in K_1,\cr
(r-4-6j_2)k&\text{if}\ k\in K_2.
\end{cases}
\]
By \eqref{class9-1}, $6j_2\equiv -4\pmod d$, hence $\pi(k)+e_kk=rk$ for $k\in\Bbb Z/d\Bbb Z$. 

System \eqref{class9-1} is equivalent to
\[
\begin{cases}
\displaystyle  1+3j_1\equiv \frac d2\pmod d,\vspace{0.3em}\cr
\displaystyle 2+3j_2\equiv 0\pmod d,
\end{cases}
\]
i.e.,
\[
\begin{cases}
\displaystyle 1+3j_1=\frac d2,\vspace{0.3em}\cr
\displaystyle 2+3j_2=d,
\end{cases}
\quad\text{or}\quad
\begin{cases}
\displaystyle 1+3j_1=\frac{5d}2,\vspace{0.3em}\cr
\displaystyle 2+3j_2=2d.
\end{cases}
\]
In the first case, $d\equiv 2\pmod 6$, whence $q=3d-1\equiv 5\pmod{18}$, and
\[
h(X)=1+aX^{(q+1)/6}+bX^{(q+1)/3}.
\]
In the second case, $d\equiv 4\pmod 6$, whence $q=3d-1\equiv 11\pmod{18}$, and
\[
h(X)=1+aX^{5(q+1)/6}+bX^{2(q+1)/3}.
\] 
\end{proof}

\begin{rmk}\label{R4.11}\rm 
In Lemmas~\ref{L4.8} -- \ref{L4.10}, it is easy to see that the polynomial $h(X)$ satisfies 
\[
\text{gcd}(h(X),\,X^{q+1}-1)=1. 
\]
For example, in Lemma~\ref{L4.8}, with $q\equiv 11\pmod{18}$, we have 
\[
h(X)=1+aX^{(q+1)/3}\pm a^2X^{(q+1)/6},
\]
where $a^3=-1$. Assume to the contrary that $h(X)$ and $X^{q+1}-1$ have a common root $x\in\f_{q^2}$. Then $a^2x^{(q+1)/6}$ is a common root of  $1\pm X-X^2$ and $X^6-1$. This is impossible since $\text{gcd}(1\pm X-X^2,X^6-1)=1$.
\end{rmk}

\subsection{Four classes}\

All permutation trinomials resulting from Algorithm~\ref{Algo2.4} have been determined in Section~4.1. These permutation trinomials, excluding those that were previously known, can be categorized into four classes. Each class covers a situation described in a lemma or several lemmas in Section~4.1. Theorem~\ref{T2.3} is applied to the situation to set the conditions on the parameters. More precisely, these conditions are
\begin{itemize}
\item $\text{gcd}(r,q-1)=1$;
\item $\text{gcd}(e_k,(q+1)/d)=1$ for all $0\le k<d$;
\item $\text{gcd}(h(X),X^{q+1}-1)=1$ (cf. Remark~\ref{R2.5});
\item the map $k\mapsto \pi(k)+e_kk$ permutes $\Bbb Z/d\Bbb Z$.
\end{itemize}
In Class~4, which covers Lemmas~\ref{L4.8} -- \ref{L4.10}, the condition $\text{gcd}(e_k,(q+1)/d)=1$ ($0\le k<d$) is satisfied by the choice of $r\pmod 3$, and the condition $\text{gcd}(h(X),X^{q+1}-1)=1$ is automatically satisfied by Remark~\ref{R4.11}.

In each class, the permutation trinomial is
\[
X^rh(X^{q-1}),
\]
where
\[
h(X)=1+aX^{i_1+j_1(q+1)/d}+bX^{i_2+j_2(q+1)/d}.
\]

\medskip
{\bf Class 1.} (Case 2.1, Lemma~\ref{L4.1} (ii))

\medskip

{\bf Conditions:} $i_1=0<i_2<(q+1)/d$, $j_1=d/2$, $0\le j_2<d$, $a^{q-1}=-1$, $(1-a)/b\in\mu_{q+1}$,
\[
\text{gcd}(1+aX^{(q+1)/2}+bX^{i_2+j_2(q+1)/d},\;X^{q+1}-1)=1,
\]
$\text{gcd}(r,q-1)=1$, $\text{gcd}(r-i_2,(q+1)/d)=1$, and
\[
k\mapsto \Bigl(-j_2\frac{q+1}d+r-i_2\Bigr)k+\delta(k)v
\]
permutes $\Bbb Z/d\Bbb Z$, where
\[
\Bigl(\frac{1+a}{1-a}\Bigr)^{(q+1)/d}=\epsilon^v
\]
and 
\[
\delta(k)=\begin{cases}
0&\text{if $k$ is even},\cr
1&\text{if $k$ is odd}.
\end{cases}
\]

\medskip

{\bf PP:} $X^r(1+aX^{(q^2-1)/2}+bX^{(q-1)(i_2+j_2(q+1)/d)})$.

\medskip
{\bf Class 2.} (Case 2.2, Lemma~\ref{L4.3} (i))
\medskip

{\bf Conditions:} $i_1=0<i_2<(q+1)/d$, $j_1=d/2$, $0\le j_2<d$, $a=-1$, $(2/b)^{(q+1)/d}=\epsilon^\theta$ for some $\theta\in\Bbb Z/d\Bbb Z$,
\[
\text{gcd}(1-X^{(q+1)/2}+bX^{i_2+j_2(q+1)/d},\;X^{q+1}-1)=1,
\]
$\text{gcd}(r,q-1)=1$, $\text{gcd}(r-i_2,(q+1)/d)=\text{gcd}(r-2i_2,(q+1)/d)=1$, and
\[
k\mapsto
\begin{cases} 
\displaystyle \Bigl(-2j_2\frac{q+1}d+r-2i_2\Bigr)k+2\theta&\text{if}\ k\equiv0\pmod 2, \vspace{0.4em}\cr
\displaystyle \Bigl(-j_2\frac{q+1}d+r-i_2\Bigr)k+\theta&\text{if}\ k\not\equiv0\pmod 2
\end{cases}
\]
permutes $\Bbb Z/d\Bbb Z$.

\medskip

{\bf PP:} $X^r(1-X^{(q^2-1)/2}+bX^{(q-1)(i_2+j_2(q+1)/d)})$.

\medskip
{\bf Class 3.} (Case 2.2, Lemma~\ref{L4.3} (ii))

\medskip

{\bf Conditions:} $i_1=0<i_2<(q+1)/d$, $j_1=d/3$ or $2d/3$, $0\le j_2<d$, $a=-1$, $((1-\epsilon^{j_1})/b)^{(q+1)/d}=\epsilon^\eta$ for some $\eta\in\Bbb Z/b\Bbb Z$, 
\[
\text{gcd}(1-X^{j_1(q+1)/d}+bX^{i_2+j_2(q+1)/d},\;X^{q+1}-1)=1,
\]
$\text{gcd}(r,q-1)=1$, $\text{gcd}(r-i_2,(q+1)/d)=\text{gcd}(r-2i_2,(q+1)/d)=1$, and
\[
k\mapsto
\begin{cases} 
\displaystyle \Bigl(-2j_2\frac{q+1}d+r-2i_2\Bigr)k-j_1\frac{q+1}d+\frac{q+1}{\text{gcd}(2,d)}+2\eta&\text{if}\ k\equiv0\pmod 3, \vspace{0.4em}\cr
\displaystyle \Bigl(-j_2\frac{q+1}d+r-i_2\Bigr)k-j_1\frac{q+1}d+\frac{q+1}{\text{gcd}(2,d)}+\eta&\text{if}\ k\not\equiv0\pmod 3
\end{cases}
\]
permutes $\Bbb Z/d\Bbb Z$.

\medskip

{\bf PP:} $X^r(1-X^{j_1(q^2-1)/d}+bX^{(q-1)(i_2+j_2(q+1)/d)})$.

\medskip
\noindent{\bf Remark.} All permutation trinomials in \cite[\S 2]{Qin-Yan-AAECC-2021} are covered by Classes 1 and 2 up to equivalence. All permutation trinomials in \cite[\S 3]{Qin-Yan-AAECC-2021} are covered by Classe 3 (with even $q$) up to equivalence.

\medskip
{\bf Class 4.} (Case 3, Lemmas~\ref{L4.8} -- \ref{L4.10}) Conditions on $q$ and $r$ and the expressions of $h(X)$ in this class are given in Table~\ref{Tb-cls7}. There are six cases in Table~\ref{Tb-cls7} according to $q\pmod{18}$ and $r\pmod 3$. However, the resulting PP, $X^rh(X^{q-1})$, modulo $X^{q^2-1}-1$, has only two cases according to $q\pmod{18}$. More precisely, let $q$, $r$ and $h(X)$ be from Table~\ref{Tb-cls7} and let $m=(q^2-1)/6$. If $q\equiv 5\pmod{18}$, then
\[
X^rh(X^{q-1})\equiv uX^s(1+cX^m-c^2X^{2m})\pmod{X^{q^2-1}-1},
\]
for some $s\equiv -1\pmod 3$, $u\in\f_{q^2}^*$ and $c\in\mu_6$.
If $q\equiv 11\pmod{18}$, then
\[
X^rh(X^{q-1})\equiv uX^s(1+cX^m-c^2X^{2m})\pmod{X^{q^2-1}-1},
\]
for some $s\equiv 1\pmod 3$, $u\in\f_{q^2}^*$ and $c\in\mu_6$.

To verify the above claim, let $l=(q+1)/6$. When $q\equiv 5\pmod{18}$ and $r\equiv 0\pmod 3$,
\[
h(X)\equiv 1+aX^{l}-a^{-1}X^{5l}\equiv -a^{-1}X^{5l}(1+cX^l-c^2X^{2l})\pmod{X^{q+1}-1},
\]
where $c=-a\in\mu_6$. Hence
\[
X^rh(X^{q-1})\equiv -a^{-1}X^{r+5l(q-1)}(1+cX^m-c^2X^{2m})\pmod{X^{q^2-1}-1},
\]
where $r+5l(q-1)\equiv -1\pmod 3$.

When $q\equiv 5\pmod{18}$ and $r\equiv 1\pmod 3$,
\[
h(X)\equiv 1+aX^{4l}\pm a^{2}X^{5l}\equiv aX^{4l}(1+cX^l-c^2X^{2l})\pmod{X^{q+1}-1},
\]
where $c=\pm a\in\mu_6$. Hence
\[
X^rh(X^{q-1})\equiv aX^{r+4l(q-1)}(1+cX^m-c^2X^{2m})\pmod{X^{q^2-1}-1},
\]
where $r+4l(q-1)\equiv -1\pmod 3$.

For the remaining cases in Table~\ref{Tb-cls7}, the claim is verified similarly. 

\begin{table}[ht]
\caption{$q$, $r$ and $h(X)$ in Class~7}\label{Tb-cls7}
   \renewcommand*{\arraystretch}{1.4}
   \vskip-1em
    \centering
     \begin{tabular}{l|l|l} 
         \hline
         \hfil $q$ & \hfil $r$  & \hfil $h(X)$ \\ \hline
         & $r\equiv 0\pmod 3$ & $h(X)=1+aX^{(q+1)/6}- a^{-1}X^{5(q+1)/6}$\\
         & $\text{gcd}(r,q-1)=1$ & $a^3=\pm1$ \\ \cline{2-3}
         $q\equiv 5\pmod{18}$ & $r\equiv 1\pmod 3$ & $h(X)=1+aX^{2(q+1)/3}\pm a^{2}X^{5(q+1)/6}$\\
         & $\text{gcd}(r,q-1)=1$ & $a^3=-1$ \\ \cline{2-3}
         & $r\equiv -1\pmod 3$ & $h(X)=1\pm b^2X^{(q+1)/6}+bX^{(q+1)/3}$\\
         & $\text{gcd}(r,q-1)=1$ & $b^3=-1$ \\
         \hline
         & $r\equiv 0\pmod 3$ & $h(X)=1+aX^{5(q+1)/6}- a^{-1}X^{(q+1)/6}$\\
         & $\text{gcd}(r,q-1)=1$ & $a^3=\pm1$ \\ \cline{2-3}
         $q\equiv 11\pmod{18}$ & $r\equiv 1\pmod 3$ & $h(X)=1+aX^{(q+1)/3}\pm a^{2}X^{(q+1)/6}$\\
         & $\text{gcd}(r,q-1)=1$ & $a^3=-1$ \\ \cline{2-3}
         & $r\equiv -1\pmod 3$ & $h(X)=1\pm b^2X^{(5q+1)/6}+bX^{2(q+1)/3}$\\
         & $\text{gcd}(r,q-1)=1$ & $b^3=-1$ \\
         \hline
     \end{tabular}
\end{table}

\subsection{Examples}\

We give an example in each of the first three classes in Section~4.2. (Note that Class~4 is already explicit.) These are rather simple examples and their primary purpose is to show that none of these classes is empty. Interested readers may explore more elaborate examples as they wish.

\begin{exmp}[Class 1]\rm
Let $q\equiv 1\pmod 4$, $d=2$, $i_1=0$, $i_2=1$, $j_1=1$, $j_2=0$. Let $a\in\f_{q^2}^*$ be such that $a^{q-1}=-1$ and
\[
\Bigl(\frac{1+a}{1-a}\Bigr)^{(q+1)/2}=-1.
\]
To see that such $a$ exists, first choose $a_0\in\f_{q^2}^*$ such that $a_0^{q-1}=-1$ and let $a=ta_0$, $t\in\f_q^*$. Since $((1+a)/(1-a))^{q+1}=1$, we have $((1+a)/(1-a))^{(q+1)/2}=\pm1$, i.e.,
\[
\Bigl(\frac{1+ta_0}{1-ta_0}\Bigr)^{(q+1)/2}=\pm 1.
\]
The equation
\[
\Bigl(\frac{1+ta_0}{1-ta_0}\Bigr)^{(q+1)/2}=1
\]
has at most $(q+1)/2$ solutions for $t$, where $(q+1)/2<q-1$. Hence there exists $t\in\f_q^*$ such that
\[
\Bigl(\frac{1+ta_0}{1-ta_0}\Bigr)^{(q+1)/2}=-1.
\]
Let $b\in\f_{q^2}^*$ be such that $(1-a)/b=-1$. Assume that $\text{gcd}(r,q-1)=1$ and $\text{gcd}(r-1,(q+1)/2)=1$ ($q=5$ and $r=3$  satisfy these conditions). We have 
\[
h(X)=1+aX^{i_1+j_1(q+1)/2}+bX^{i_2+j_2(q+1)/2}=1+aX^{(q+1)/2}-(1-a)X.
\]
We claim that $\text{gcd}(h(X),X^{q+1}-1)=1$. Assume to the  contrary that $h(X)$ and $X^{q+1}-1$ have a common root $x$. Then $x^{(q+1)/2}=\pm 1$. If $x^{(q+1)/2}=1$, then $x=(1+a)/(1-a)$, whence $x^{(q+1)/2}=((1+a)/(1-a))^{(q+1)/2}=-1$, which is a contradiction. If $x^{(q+1)/2}=-1$, then $x=(1-a)/(1-a)=1$, which is also a contradiction.

In the notation of Class 1,
\[
k\mapsto \Bigl(-j_2\frac{q+1}d+r-i_2\Bigr)k+\delta(k)v=k,
\]
which permutes $\Bbb Z/2\Bbb Z$. Therefore, 
\[
X^rh(X^{q-1})=X^r(1+aX^{(q^2-1)/2}+(1-a)X^{q-1})
\]
is a PP of $\f_{q^2}$.
\end{exmp}

\begin{exmp}[Class 2]\rm
Let $q\equiv 1\pmod 4$, $d=2$, $i_1=0$, $0<i_2<(q+1)/2$, $i_2$ even, $j_1=1$, $j_2=0$. Let $a=-1$ and $b\in\f_{q^2}^*$ be such that $(2/b)^{(q+1)/2}=1$. Assume that $\text{gcd}(r,q-1)=1$, $\text{gcd}(r-i_2,(q+1)/2)=1$ and $\text{gcd}(r-2i_2,(q+1)/2)=1$ ($q=5$, $r=3$ and $i_2=2$ satisfy these conditions). We have
\[
h(X)=1-X^{(q+1)/2}+bX^{i_2}.
\]
We claim that $\text{gcd}(h(X),X^{q+1}-1)=1$. Assume to the contrary that $h(X)$ and $X^{q+1}-1$ have a common root $x$. Then $x^{(q+1)/2}=\pm 1$. If $x^{(q+1)/2}=1$, then $0=h(x)=bx^{i_2}$, which is a contradiction. If $x^{(q+1)/2}=-1$, then $x^{i_2}=-2/b$, whence $1=(x^{(q+1)/2})^{i_2}=(x^{i_2})^{(q+1)/2}=(-2/b)^{(q+1)/2}=-1$, which is also a contradiction.

In the notation of Class 2,
\begin{align*}
k\mapsto\,&
\begin{cases} 
\displaystyle \Bigl(-2j_2\frac{q+1}d+r-2i_2\Bigr)k+2\theta&\text{if}\ k\equiv0\pmod 2 \vspace{0.4em}\cr
\displaystyle \Bigl(-j_2\frac{q+1}d+r-i_2\Bigr)k+\theta&\text{if}\ k\not\equiv0\pmod 2
\end{cases}\cr
&=k,
\end{align*}
which permutes $\Bbb Z/2\Bbb Z$. Therefore, 
\[
X^rh(X^{q-1})=X^r(1-X^{(q^2-1)/2}+bX^{i_2(q-1)})
\]
is a PP of $\f_{q^2}$.
\end{exmp}

\begin{exmp}[Class 3]\rm
Let $q=2^{2n+1}$, $d=3$, $i_1=0$, $0<i_2<(q+1)/3$, $j_1=1$, $j_2=0$. Let $a=-1$ and $b\in\f_{q^2}^*$ be such that $((1-\epsilon)/b)^{(q+1)/3}=1$, where $\epsilon$ is an element of $\f_{q^2}^*$ of order $3$. Assume that $i_2\not\equiv 0,r,(q+1)/3\pmod 3$, $\text{gcd}(r,q-1)=1$, $\text{gcd}(r-i_2,(q+1)/3)=1$ and $\text{gcd}(r-2i_2,(q+1)/3)=1$  ($q=2^3$, $r=3$ and $i_2=1$ satisfy these conditions). We have
\[
h(X)=1-X^{(q+1)/3}+bX^{i_2}.
\]
We claim that $\text{gcd}(h(X),X^{q+1}-1)=1$. Assume to the contrary that $h(X)$ and $X^{q+1}-1$ have a common root $x$. Then $x^{(q+1)/3}=1$, $\epsilon$ or $\epsilon^{-1}$. If $x^{(q+1)/3}=1$, then $0=h(x)=bx^{i_2}$, which is a contradiction. If $x^{(q+1)/3}=\epsilon$, then $x^{i_2}=(1-\epsilon)/b$, whence $\epsilon^{i_2}=(x^{i_2})^{(q+1)/3}=((1-\epsilon)/b)^{(q+1)/3}=1$, which is a contradiction. If $x^{(q+1)/3}=\epsilon^{-1}$, then $x^{i_2}=(1-\epsilon^{-1})/b=\epsilon^{-1}(1-\epsilon)/b$, whence $\epsilon^{-i_2}=\epsilon^{-(q+1)/3}$, which is also a contradiction.

In the notation of Class 3,
\begin{align*}
k\mapsto\,&
\begin{cases} 
\displaystyle \Bigl(-2j_2\frac{q+1}d+r-2i_2\Bigr)k-j_1\frac{q+1}d+\frac{q+1}{\text{gcd}(2,d)}+2\eta&\text{if}\ k\equiv0\pmod 3 \vspace{0.4em}\cr
\displaystyle \Bigl(-j_2\frac{q+1}d+r-i_2\Bigr)k-j_1\frac{q+1}d+\frac{q+1}{\text{gcd}(2,d)}+\eta&\text{if}\ k\not\equiv0\pmod 3
\end{cases}\cr
&=(r-i_2)k-j_1\frac{q+1}3,
\end{align*}
which permutes $\Bbb Z/3\Bbb Z$. Therefore, 
\[
X^rh(X^{q-1})=X^r(1-X^{(q^2-1)/3}+bX^{i_2(q-1)})
\]
is a PP of $\f_{q^2}$.
\end{exmp}

\section{Additional Examples}

In this section we give a few examples using the forward approach of Algorithm~\ref{Algo2.4}. We continue to follow the notation of Algorithm~\ref{Algo2.4}.

\begin{exmp}\label{E5.1}\rm
Let $r$ and $q$ be such that $\text{gcd}(r,q-1)=1$. Let $d=1$, $s_0=0$ (i.e., $h(0)\ne 0$), $0\le t_0<q+1$, $\tau_0=0$, and  $e_0=r-t_0$ be such that $\text{gcd}(e_0,q+1)=1$. Let $h\in\mathcal L_0(t_0,0;\lambda_0)$, that is, $h\in\f_{q^2}[X]$ is a self-dual polynomial of degree $t_0$ such that $\text{gcd}(h,X^{q+1}-1)=1$. Then $X^rh(X^{q-1})$ is a PP of $\f_{q^2}$. This is the PP in \cite[Theorem~5.1]{Zieve-arXiv1310.0776}.
\end{exmp}

\begin{exmp}\label{E4.2}\rm
Let $r$ and $q$ be such that $\text{gcd}(r,q-1)=1$. Let $d=1$, $s_0=0$ (i.e., $h(0)\ne 0$), $0\le t_0<q+1$, $q+1-t_0\le \tau_0\le t_0$ and $e_0=r-t_0+\tau_0$ be such that $\text{gcd}(e_0,q+1)=1$. Let $h\in\mathcal L_0(t_0,\tau_0;1)$, that is, $h=P+X^{q+1-\tau_0}Q$, where $P,Q\in\f_{q^2}[X]$, $\deg P=t_0-\tau_0$, $\tilde P=P$, $\deg Q=\tau_0+t_0-(q+1)$, $\tilde Q=Q$, $\text{gcd}(h,X^{q+1}-1)=1$. Then $X^rh(X^{q-1})$ is a PP of $\f_{q^2}$. This construction does not seem to have appeared in the literature.
\end{exmp}

As an explicit instance of Example~\ref{E4.2}, let's consider the following situation: Let $q\ge 5$ be odd, $t_0=(q+5)/2$, $\tau_0=(q+1)/2$, $\text{gcd}(r,q-1)=\text{gcd}(r-2,q+1)=1$, $P=1+X^2$, $Q=a^q+aX^2$, where $a\in\f_{q^2}^*$ is such that $(1+a)^{(q^2-1)/2}=(1-a)^{(q^2-1)/2}=(-1)^{(q-1)/2}$. The number of such elements $a$ is $(q^2-1)/4$; see Lemma~\ref{LL4.3} below. Let
\[
h(X)=P(X)+X^{(q+1)/2}Q(X)=1+X^2+X^{(q+1)/2}(a^q+aX^2).
\]
We claim that $\text{gcd}(h(X),X^{q+1}-1)=1$. Assume to the contrary that $h(X)$ and $X^{q+1}-1$ have a common root $x$. Then
\[
x^{(q+1)/2}=-\frac{1+x^2}{a^q+ax^2},
\]
whence
\[
1=x^{q+1}=\Bigl(\frac{1+x^2}{a^q+ax^2}\Bigr)^2.
\]
Thus
\[
\frac{1+x^2}{a^q+ax^2}=\pm1,
\]
giving
\[
x^2=-\frac{1\mp a^q}{1\mp a}=-(1\mp a)^{q-1}.
\]
Therefore,
\[
1=(x^2)^{(q+1)/2}=\bigl(-(1\mp a)^{q-1}\bigr)^{(q+1)/2}=(-1)^{(q+1)/2}(1\mp a)^{(q^2-1)/2}=-1,
\]
which is a contradiction. 

Therefore,
\[
X^rh(X^{q-1})=X^r(1+X^{2(q-1)}+a^qX^{(q^2-1)/2}+aX^{(q+5)(q-1)})
\]
is a PP of $\f_{q^2}$.

\begin{lem}\label{LL4.3}
Let $q$ be odd and 
\[
\mathcal A=\{a\in\f_{q^2}^*: (1+a)^{(q^2-1)/2}=(1-a)^{(q^2-1)/2}=(-1)^{(q-1)/2}\}.
\]
Then $|\mathcal A|=(q^2-1)/4$.
\end{lem}

\begin{proof}
Choose $u\in\f_{q^2}^*$ such that $u^{(q^2-1)/2}=(-1)^{(q-1)/2}$, and let 
\[
\mathcal X=\{(x,y)\in\f_{q^2}^2: x^2+y^2=2u\}.
\]
By \cite[Lemma~6.55]{Hou-ams-gsm-2018} or \cite[Lemma~6.24]{Lidl-Niederreiter-FF-1997}, $|\mathcal X|=q^2-1$. Note that for $a\in\f_{q^2}$,
\begin{align*}
&(1+a)^{(q^2-1)/2}=(1-a)^{(q^2-1)/2}=(-1)^{(q-1)/2}\cr
\Leftrightarrow\ & 
\begin{cases}
1+a=u^{-1}x^2\cr
1-a=u^{-1}y^2
\end{cases}\ \text{for some}\ (x,y)\in\mathcal X.
\end{align*}
If $u$ is a nonsquare of $\f_{q^2}$, then
\[
|\mathcal A|=\frac 14|\mathcal X|=\frac14(q^2-1).
\]
If $u$ is a square of $\f_{q^2}$, partition $\mathcal X$ as $\mathcal X=\mathcal X_1\sqcup\mathcal X_2\sqcup\mathcal X_3$, where
\begin{align*}
\mathcal X_1\,&=\{(x,y)\in\mathcal X: x^2\ne 0, u;\ y^2\ne 0,u\},\cr
\mathcal X_2\,&=\{(x,y)\in\mathcal X: x^2=y^2=u\},\cr
\mathcal X_3\,&=\{(x,y)\in\mathcal X: x=0\ \text{or}\ y=0\}.
\end{align*}
Then $|\mathcal X_1|+|\mathcal X_2|+|\mathcal X_3|=q^2-1$, $|\mathcal X_2|=4$ and $|\mathcal X_3|=4$. Hence
\[
|\mathcal A|=\frac 14|\mathcal X_1|+\frac12|\mathcal X_3|=\frac14(q^2-1-8)+2=\frac 14(q^2-1).
\]
\end{proof}

We conclude this section with a random concrete example.

\begin{exmp}\label{E5.4}\rm
Let $q=47$, $d=6$, $r=3$, so $(q+1)/d=8$. Let $\gamma$ be a primitive element of $\f_{47^2}$ with minimal polynomial $X^2+X+13$ over $\f_{47}$ and let $\epsilon=\gamma^{(q^2-1)/d}=\gamma^{46\cdot 8}$. Choose sequences $s_k$, $t_k$, $\tau_k$, $\pi(k)$, and $\lambda_k$ as shown in Table~\ref{Tb1}.

\begin{table}[ht]
\caption{Sequences in Example~\ref{E5.4}}\label{Tb1}
   \renewcommand*{\arraystretch}{1.4}
   \vskip-1em
    \centering
     \begin{tabular}{c|ccccccc}  
         \hline
         $k$  & $s_k$ & $t_k$ & $\tau_k$ & $e_k=r-2s_k-t_k+\tau_k$ & $e_kk$ & $\pi(k)$ & $\lambda_k$ \\    \hline
		$0$ & $0$ & $0$ & $0$ & $3$ & $0$ & $0$ & $1$ \\
		$1$ & $1$ & $5$ & $5$ & $1$ & $1$ & $2$ & $\gamma^{46\cdot 2}$ \\
		$2$ & $2$ & $2$ & $0$ & $-3$ & $0$ & $1$ & $\gamma^{46}$ \\
		$3$ & $0$ & $5$ & $3$ & $1$ & $3$ & $1$ & $\gamma^{46}$ \\
		$4$ & $3$ & $0$ & $0$ & $-3$ & $0$ & $2$ & $\gamma^{46\cdot 2}$ \\
		$5$ & $0$ & $7$ & $1$ & $-3$ & $3$ & $2$ & $\gamma^{46\cdot 2}$ \\
           \hline
     \end{tabular}
\end{table}

Choose 
\begin{align*}
&L_0=1\in\mathcal L_0(0,0;1),\cr
&L_1=\gamma^2+X^3(\gamma^{46\cdot 38}+X^2)\in\mathcal L_1(5,5;\gamma^{46\cdot 2}),\cr
&L_2=\gamma^{46\cdot 47}+X^2\in\mathcal L_2(2,0;\gamma^{46}),\cr
&L_3=\gamma^{46\cdot 47}+X^2+X^5\cdot\gamma^{25}\in\mathcal L_3(5,3;\gamma^{46}),\cr
&L_4=\gamma^2\in\mathcal L_4(0,0;\gamma^{46\cdot 2}),\cr
&L_5=\gamma^{46\cdot 46}+X^6+X^7\cdot\gamma^{42}\in\mathcal L_5(7,1;\gamma^{46\cdot 2}).
\end{align*}
We have
\begin{align*}
&\sum_iM_{i0}X^i=L_0=1,\cr
&\sum_iM_{i1}X^i=XL_1=\gamma^2X+\gamma^{46\cdot 38}X^4+X^6,\cr
&\sum_iM_{i2}X^i=X^2L_2=\gamma^{46\cdot 47}X^2+X^4,\cr
&\sum_iM_{i3}X^i=L_3=\gamma^{46\cdot 47}+X^2+\gamma^{25}X^5,\cr
&\sum_iM_{i4}X^i=X^3L_4=\gamma^2X^3,\cr
&\sum_iM_{i5}X^i=L_5=\gamma^{46\cdot 46}+X^6+\gamma^{42}X^7.
\end{align*}
Hence
\[
[M_{ik}]=\left[\begin{matrix}
1&0&0&\gamma^{46\cdot 47}&0&\gamma^{46\cdot 46}\cr
0&\gamma^2&0&0&0&0\cr
0&0&\gamma^{46\cdot 47}&1&0&0\cr
0&0&0&0&\gamma^2&0\cr
0&\gamma^{46\cdot 38}&1&0&0&0\cr
0&0&0&\gamma^{25}&0&0\cr
0&1&0&0&0&1\cr
0&0&0&0&0&\gamma^{42}
\end{matrix}\right]
\]
and
\begin{align*}
[a_{ij}]\,&=\frac 16[M_{ik}][\gamma^{-46\cdot 8kj}]\cr
&=\left[
\begin{matrix}
 11+37 \gamma & 11+25 \gamma & 31+44 \gamma & 5+10 \gamma & 5+22 \gamma & 32+3 \gamma \\
 37+39 \gamma & 12+32 \gamma & 22+40 \gamma & 10+8 \gamma & 35+15 \gamma & 25+7 \gamma \\
 3 \gamma & 19+32 \gamma & 36+12 \gamma & 31+3 \gamma & 35+32 \gamma & 20+12 \gamma \\
 37+39 \gamma & 35+15 \gamma & 22+40 \gamma & 37+39 \gamma & 35+15 \gamma & 22+40 \gamma \\
 24+6 \gamma & 24+18 \gamma & 9+14 \gamma & 39+41 \gamma & 14+27 \gamma & 31+35 \gamma \\
 43+36 \gamma & 4+11 \gamma & 43+36 \gamma & 4+11 \gamma & 43+36 \gamma & 4+11 \gamma \\
 16 & 8 & 39 & 31 & 39 & 8 \\
 42+16 \gamma & 27+24 \gamma & 32+8 \gamma & 5+31 \gamma & 20+23 \gamma & 15+39 \gamma \\
\end{matrix}\right].
\end{align*}
In conclusion,
\[
X^3\sum_{\substack{0\le i<8\cr 0\le j<6}}a_{ij}X^{46(i+8j)}
\]
is a PP of $\f_{47^2}$.
\end{exmp}




\begin{thebibliography}{99}

\bibitem{Cao-Hou-Mi-Xu-FFA-2020}
X. Cao, X. Hou, J. Mi, S. Xu,
{\it More permutation polynomials with Niho exponents which permute $\Bbb F_{p^2}$}, Finite Fields Appl.  {\bf 62} (2020), Article 101626.

\bibitem{Fernando-Hou-FFA-2012}
N. Fernando and X. Hou,
{\it A piecewise construction of permutation polynomials over finite fields}, Finite Fields Appl. {\bf 18} (2012), 1184 -- 1194.


\bibitem{Hou-ams-gsm-2018}
X. Hou, {\it Lectures on Finite Fields}, Graduate Studies in Mathematics, vol. 190, American Mathematical Society, Providence, RI, 2018.

\bibitem{Hou-pp}
X. Hou, {\it Number of equivalence classes of rational functions over finite fields}, preprint.

\bibitem{Lavorante-arXiv:2105.12012}
V. P. Lavorante, {\it New families of permutation trinomials constructed by permutations of $\mu_{q+1}$}, arXiv:2105.12012.

\bibitem{Li-Qu-Chen-Li-CC-2018}
K. Li, L. Qu, X. Chen, C. Li, {\it Permutation polynomials of the form $cx+\text{Tr}_{q^l/q}(x^a)$ and permutation trinomials over finite fields with even characteristic}, Cryptogr. Commun. {\bf 10} (2018), 531 -- 554.

\bibitem{Lidl-Niederreiter-FF-1997}
R. Lidl and H. Niederreiter, {\it Finite Fields},  Cambridge University Press, Cambridge, 1997.


\bibitem{Niederreiter-Winterhof-DM-2005}
H. Niederreiter and A. Winterhof, {\it Cyclotomic $\mathcal R$-orthomorphisms of finite fields}, Discrete Math. {\bf 295} (2005), 161 -- 171.


\bibitem{Park-Lee-BAMS-2001} Y. H. Park and J. B. Lee, {\it Permutation polynomials and group permutation polynomials}, Bull. Austral. Math. Soc. {\bf 63} (2001), 67 -- 74.

\bibitem{Qin-Yan-AAECC-2021}
X. Qin and L. Yan, {\it Constructing permutation trinomials via monomials on the subsets of $\mu_{q+1}$}, AAECC, Published on April 10, 2021, DOI:10.1007/s00200-021-00505-8.

\bibitem{Wang-LNCS-2007}
Q. Wang, {\it Cyclotomic mapping permutation polynomials over finite fields}, in {\it Sequences, Subsequences, and Consequences}, S.W. Golomb, G. Gong, T. Helleseth, H.-Y. Song, (Eds.), pp. 119 -- 128, Lecture Notes in Comput. Sci., vol. 4893, Springer, Berlin, 2007.

\bibitem{Wang-FFA-2013}
Q. Wang, {\it Cyclotomy and permutation polynomials of large indices}, Finite Fields Appl. {\bf 22} (2013), 57 -- 69.

\bibitem{Zha-Hu-FFA-2012}
Z. Zha and  L. Hu,
{\it Two classes of permutation polynomials over finite fields}, Finite Fields Appl. {\bf 18} (2012), 781 -- 790.

\bibitem{Zheng-Yuan-Yu-FFA-2019}
D. Zheng, M. Yuan, L. Yu, {\it Two types of permutation polynomials with special forms}, Finite Fields Appl. {\bf 56} (2019), 1 -- 16.


\bibitem{Zieve-PAMS-2009}
M. E. Zieve, {\it On some permutation polynomials over $\Bbb F_q$  of the form $x^r h(x^{(q-1)/d})$}, Proc. Amer. Math. Soc. {\bf 137} (2009), 2209 -- 2216.

\bibitem{Zieve-arXiv1310.0776}
M. E. Zieve, {\it Permutation polynomials on $\Bbb F_q$ induced from R\'edei function bijections on subgroups of $\Bbb F_q^*$}, arXiv:1310.0776.




\end{thebibliography}
\end{document}